\documentclass[a4paper,11pt,fullpage]{article}
\usepackage[english]{babel}

\usepackage[left=1in,right=1in,top=1in,bottom=1in]{geometry}

\usepackage{amsmath}
\usepackage{amsthm}
\usepackage{amsfonts}
\usepackage{amssymb}
\usepackage{amsxtra}
\usepackage{latexsym}
\usepackage{ifthen}
\usepackage{cite}

\usepackage{epic}
\usepackage{eepic}

\numberwithin{equation}{section}
\newtheorem{theorem}{Theorem}[section]
\newtheorem{lemma}{Lemma}[section]
\newtheorem{remark}{Remark}[section]
\newtheorem{definition}{Definition}[section]

\newtheorem{example}{Example}[section]

\def\XXint#1#2#3{{\setbox0=\hbox{$#1{#2#3}{\int}$}
     \vcenter{\hbox{$#2#3$}}\kern-.5\wd0}}

\begin{document}

\title{ $\mathcal{B}_{1}$ classes of De\,Giorgi-Ladyzhenskaya-Ural'tseva and
their applications to elliptic and parabolic equations with generalized Orlicz
growth conditions
}

\author{Igor I. Skrypnik, Mykhailo V. Voitovych 
 }

 \maketitle

\begin{abstract}
We introduce elliptic and parabolic $\mathcal{B}_{1}$ classes that generalize
the well-known $\mathfrak{B}_{p}$ classes of De\,Giorgi, Ladyzhenskaya and Ural'tseva
with $p>1$.
New classes are applied to prove pointwise continuity of solutions of elliptic and parabolic
equations with nonstandard growth conditions.
Our considerations cover new cases of variable exponent and
$(p, q)$-phase growth including the ,,singular-degenerate'' parabolic case $p<2<q$.




\end{abstract}

\pagestyle{myheadings} \thispagestyle{plain} \markboth{Igor I.
Skrypnik}{$\mathcal{B}_{1}$ classes of De\,Giorgi-Ladyzhenskaya-Ural'tseva and
their applications . . .}

\section{Introduction}\label{Introduction}

In the present paper we consider the question of regularity of weak solutions
to quasilinear elliptic and parabolic equations with generalized Orlicz growth.
In addition, we define elliptic and parabolic $\mathcal{B}_{1}$ classes that
generalize the well-known $\mathfrak{B}_{p}$ classes of De Giorgi, Ladyzhenskaya
and Ural'tseva and prove the continuity of the functions belonging to these classes.
Moreover, it turns out that solutions of the following equations belong to the
corresponding $\mathcal{B}_{1}$ classes:
\begin{equation}\label{eqSkr1.1}
{\rm div}\bigg( g_{i}(x,|\nabla u|)\frac{\nabla u}{|\nabla u|} \bigg)=0, \quad
i=\overline{1,4},
\end{equation}
\begin{equation}\label{eqSkr1.2}
u_{t}-{\rm div}\bigg( g_{i}(x,t,|\nabla u|)\frac{\nabla u}{|\nabla u|} \bigg)=0,\
\quad i=\overline{1,4},
\end{equation}
$$
g_{1}(\cdot, {\rm v}):= {\rm v}^{p(\cdot)-1}+{\rm v}^{q(\cdot)-1},
\quad
g_{2}(\cdot, {\rm v}):={\rm v}^{p(\cdot)-1} \big(1+\ln(1+{\rm v}) \big),
\quad {\rm v}>0,
$$
$$
g_{3}(\cdot, {\rm v}):= {\rm v}^{p-1}+a(\cdot){\rm v}^{q-1}, \
a(\cdot)\geqslant 0, \quad
g_{4}(\cdot, {\rm v}):={\rm v}^{p-1} \big(1+b(\cdot)\ln(1+{\rm v}) \big), \
b(\cdot)\geqslant 0, \  {\rm v}>0,
$$
where $p(\cdot)$, $q(\cdot)$, $a(\cdot)$ and $b(\cdot)$ satisfy the inequalities
\begin{equation}\label{eqSkr1.3}
|p(z)-p(y)|+|q(z)-q(y)|\leqslant
\frac{\lambda(|z-y|)}
{\big|\ln |z-y|\big|},
\quad z\neq y,
\end{equation}
$$
|a(z)-a(y)|\leqslant a_{0}|z-y|^{\alpha}\, e^{\lambda(|z-y|)},
\quad z\neq y, \quad a_{0}>0, \quad 0<\alpha \leqslant 1,
\quad \lim\limits_{r\rightarrow 0}r^{\alpha}\,e^{\lambda(r)}=0,
$$
$$
|b(z)-b(y)|\leqslant  \frac{b_{0}\,e^{\lambda(|z-y|)}}{\big|\ln |z-y| \big|},
\quad z\neq y, \quad b_{0}>0,
\quad \lim\limits_{r\rightarrow 0} \frac{e^{\lambda(r)}}{\ln r^{-1}}=0,
$$
with a precise choice of $\lambda(\cdot)\geqslant 0$.

The study of regularity of minima of functionals with non-standard growth of
$(p,q)$-type has been initiated by Zhikov
\cite{ZhikIzv1983, ZhikIzv1986, ZhikJMathPh94, ZhikJMathPh9798, ZhikKozlOlein94},
Marcellini \cite{Marcellini1989, Marcellini1991} and Lieberman \cite{Lieberman91},
and in the last
thirty years, the qualitative theory of second order equations with so-called log-condition
(i.e. if $\lambda(r)\leqslant L_{1}<+\infty$) has been actively developed (see, for instance,
\cite{AcerbiFuscoJDE94, AcerbiMingioneArchRat01, AcerbiMingioneAnSc01, AcerbiMingioneArchRat02,
AcerbiMingioneJRAngMath05, Alhutov97, AlhutovKrash04, AlhutovKrash08, AlkhSurnAlgAn19,
AlkhSurnApAn19, BarColMing, BarColMingStPt16, BarColMingCalc.Var.18, BenedMascoloAbsApplAn04, BurSkr_JEvolEq, ByunOh17, ByunRyuShin18, ChiadoPiatCoscia, ColMing15, ColMing218, ColMingJFnctAn16, ElMarcMas16, ElMarcMasPuraAppl16, ElMarcMasAdvCalc17, GiandiNap13, HarHastLee18, HarHastToiv17, WangLiuZhao19, Ok16, OkJMAnAppl16, OkJMAdvNonAn18, OkCalcVar17, ZhangRadul18,
ZhikIzv1983, ZhikIzv1986, ZhikJDiffEq91, ZhikJMathPh94}).
Equations of this type and  systems of such equations arise in various problems
of mathematical physics, a description of which can be found in the monographs of
Antontsev, D\'{\i}az, Shmarev \cite{AntDiazShm2002_monogr},
Harjulehto, H\"{a}st\"{o} \cite{HarHastOrlicz},
R\r{u}\v{z}i\v{c}ka \cite{Ruzicka2000} and Weickert \cite{Weickert}.

The case when condition \eqref{eqSkr1.3} holds, differs substantially
from the log-case, i.e. when $\lambda(r)\leqslant L_{1}<+\infty$.
To our knowledge there are few results in this direction. Zhikov \cite{ZhikPOMI04}
obtained a generalization of the logarithmic condition which guarantees the density
of smooth functions in Sobolev space $W^{1,p(x)}(\Omega)$. Particularly,
this result holds if $1<p\leqslant p(x)$,
$$
|p(x)-p(y)|\leqslant
\frac{L \left|\ln \left|\ln \frac{1}{|x-y|}\right| \right|}{\ln \frac{1}{|x-y|}},
\quad x\neq y, \quad L<p/n.
$$
Interior continuity and continuity up to the boundary for $p(x)$-Laplace equation
were proved by Alkhutov and Surnachev \cite{AlkhSurnAlgAn19}. It is natural
to conjecture that the interior continuity holds for elliptic and parabolic equations
of the form \eqref{eqSkr1.1}, \eqref{eqSkr1.2} under the condition \eqref{eqSkr1.3}.

In the present paper has been made an attempt to unify the approach of De\,Giorgi
to establish the local regularity of solutions to elliptic and parabolic equations
with non-standard growth. As it was already mentioned, we give an extension of the
well known elliptic and parabolic $\mathfrak{B}_{p}$ classes defined by
Ladyzhenskaya and Ural'tseva \cite{LadUr} and Di\,Benedetto
\cite{DiBenedetto86, DiBenedettoDegParEq}.

Our paper is organized as follows. In Section 2 we define elliptic $\mathcal{B}_{1}$
classes and prove local continuity for functions of $\mathcal{B}_{1}$
classes. In Section 3 we define parabolic $\mathcal{B}_{1}$ classes and prove
pointwise continuity
of functions from these classes. Moreover, we give an answer to still open problem
on the regularity of solutions to parabolic equations with $(p,q)$-growth in the case
$p<2<q$.
An approach to its solution was announced in our review article \cite{SkrVoitUMB19}.
For other well-known cases when $1<p\leqslant q\leqslant 2$ or $2\leqslant p\leqslant q<+\infty$,
we refer the reader to the papers of Hwang and Lieberman
\cite{HwangLieberman287, HwangLieberman288}.
We also note that our proofs do not require studying the special properties
of Orlicz spaces (cf. \cite{HarHastLee18, HarHastToiv17, HastOkarXiv19, WangLiuZhao19}).


\section{Elliptic $\mathcal{B}_{1}$ classes}\label{sectellB1cl}

\subsection{Notation and auxiliary propositions}

Everywhere below, $\Omega$ is a bounded domain in $\mathbb{R}^{n}$, $n\geqslant2$.
For arbitrary $\rho>0$ and $y\in \mathbb{R}^{n}$,
$B_{\rho}(y):=\{x\in \mathbb{R}^{n}: |x-y|<\rho\}$ is the open $n$-dimensional ball
centered at $y$ and radius $\rho$.
For every Lebesgue measurable set $E\subset \mathbb{R}^{n}$, we denote by $|E|$ the $n$-dimensional
Lebesgue measure of $E$ (or the $n+1$-dimensional measure, if $E\subset \mathbb{R}^{n+1}$).
We will also use the well-known notation for sets, function spaces and for their elements
(see \cite{DiBenedettoDegParEq, LadUr} for references). 

The following two lemmas will be used in the sequel.
The first one is the well-known De\,Giorgi-Poincare lemma (see \cite[Chap.\,II,\,Lemma\,3.9]{LadUr}).
\begin{lemma}\label{lem2.1DGPn}
Let $u\in W^{1,1}(B_{\rho}(x_{0}))$.
Then, for any $k,l\in \mathbb{R}$, $k<l$, the following
inequalities hold:
$$
(l-k)\,|A^{+}_{l,\rho}|^{1-\frac{1}{n}}
\leqslant
\frac{c\,\rho^{n}}{|B_{\rho}(x_{0})\setminus A^{+}_{k,\rho}|}
\int\limits_{A^{+}_{k,\rho}\setminus A^{+}_{l,\rho}} |\nabla u|\,dx,
$$
\begin{equation}\label{eq2.2}
(l-k)\,|A^{-}_{k,\rho}|^{1-\frac{1}{n}}
\leqslant
\frac{c\,\rho^{n}}{|B_{\rho}(x_{0})\setminus A^{-}_{l,\rho}|}
\int\limits_{A^{-}_{l,\rho}\setminus A^{-}_{k,\rho}} |\nabla u|\,dx,
\end{equation}
where $A^{+}_{k,\rho}:= B_{\rho}(x_{0})\cap \{u>k\}$, $A^{-}_{k,\rho}:= B_{\rho}(x_{0})\cap \{u<k\}$
and $c$ is a positive constant depending only on $n$.
\end{lemma}

The following lemma can also be found in \cite[Chap.~II, Lemma~4.7]{LadUr}.
\begin{lemma}\label{lemonyj}
Let $y_{j}$, $j=0,1,2,\ldots$, be a sequence of nonnegative numbers satisfying
$$
y_{j+1}\leqslant c\,b^{j}y_{j}^{1+\delta}, \quad j=0,1,2,\ldots,
$$
with some constants $\delta>0$ and $c,b>1$. Then
$$
y_{j}\leqslant c^{\frac{(1+\delta)^{j}-1}{\delta}}b^{\frac{(1+\delta)^{j}-1}{\delta^{2}}-\frac{j}{\delta}}
y_{0}^{(1+\delta)^{j}}, \quad j=0,1,2,\ldots.
$$
Particularly, if $y_{0}\leqslant \nu:=c^{-\frac{1}{\delta}}b^{-\frac{1}{\delta^{2}}}$, then
$$
y_{j}\leqslant \nu b^{-\frac{j}{\delta}} \quad \text{and} \quad \lim_{j\rightarrow\infty} y_{j}=0.
$$
\end{lemma}

\subsection{Elliptic $\mathcal{B}_{1}$ classes and local continuity}

Let $\lambda(r)$ be continuous and non-increasing on the interval
$(0,1)$ and consider the function
$g(x, {\rm v}): \Omega\times \mathbb{R}_{+}\rightarrow \mathbb{R}_{+}$
having the following properties:
for every ${\rm v}\in \mathbb{R}_{+}$, the function
$x\rightarrow g(x, {\rm v})$ is measurable;
for every $x\in\Omega$, the function ${\rm v}\rightarrow g(x, {\rm v})$ is continuous
and nondecreasing;
$$
\lim\limits_{{\rm v}\rightarrow +0}g(x, {\rm v})=0 \ \ \text{and} \ \
\lim\limits_{{\rm v}\rightarrow +\infty}g(x, {\rm v})=+\infty.
$$
For the function $g$ we also assume that
\begin{itemize}
\item[(${\rm g}_{1}$)]
there exist $c_{1}>0$, $q>1$ and $s_{0}\geqslant0$ such that,
for a.a. $x\in \Omega$ and for ${\rm w}\geqslant{\rm v}>s_{0}$,
$$
\frac{g(x, {\rm w})}{g(x, {\rm v})}\leqslant
c_{1} \left( \frac{{\rm w}}{{\rm v}} \right)^{q-1},
$$
\end{itemize}
\begin{itemize}
\item[(${\rm g}_{2}$)]
for $K>0$ and for any ball $B_{8r}(x_{0})\subset\Omega$ there exists $c_{2}(K)>0$
such that, for a.\,a. $x_{1}$, $x_{2}\in B_{r}(x_{0})$ and for all
$r\leqslant {\rm v}\leqslant K$,
$$
g(x_{1},  {\rm v}/r)\leqslant c_{2}(K)e^{\lambda(r)}g(x_{2},  {\rm v}/r).
$$
\end{itemize}

\begin{definition}
{\rm We say that a measurable function $u:\Omega\rightarrow \mathbb{R}$
belongs to the elliptic class $\mathcal{B}_{1,g,\lambda}(\Omega)$, if
$u\in W^{1,1}_{{\rm loc}}(\Omega)\cap L^{\infty}(\Omega)$ and there exist
positive numbers $c_{3}\geqslant q$, $c_{4}$, $c_{5}$, $c_{6}$, $K_{1}$
such that for any ball $B_{\rho}(x_{0})\subset B_{8\rho}(x_{0})\subset \Omega$,
for any $k$, $l\in \mathbb{R}$, $k<l$, $|k|, |l|<M:={\rm ess}\sup\limits_{\Omega} |u|$,
for any $\varepsilon\in (0,1]$, $\sigma\in (0,1)$, and any
$\zeta\in C_{0}^{\infty}(B_{\rho}(x_{0}))$,
$0\leqslant\zeta\leqslant1$, $\zeta=1$ in $B_{\rho(1-\sigma)}(x_{0})$,
$|\nabla \zeta|\leqslant (\sigma\rho)^{-1}$,
%
the following inequalities hold:
\begin{multline}\label{eqSkrel2.1}
\int\limits_{A^{+}_{k,\rho}\setminus A^{+}_{l,\rho}}
|\nabla u|\,\zeta^{c_{3}}dx\leqslant
K_{1}\frac{e^{c_{4}\lambda(\rho)}}{\varepsilon}\, \frac{M_{+}(k,\rho)}{\rho}\,
|A^{+}_{k,\rho}\setminus A^{+}_{l,\rho}|
\\
+  \frac{K_{1}\varepsilon^{c_{5}}\sigma^{-c_{6}}}{g\left(x_{0}, \frac{M_{+}(k,\rho)}
{\rho} \right)}
\int\limits_{A^{+}_{k,\rho}} g\left( x, \frac{K_{1}(u-k)_{+}}{\sigma\rho\,\zeta} \right)
\frac{(u-k)_{+}}{\rho}\, \zeta^{c_{3}-1}dx,
\end{multline}
\begin{multline}\label{eqSkrel2.2}
\int\limits_{A^{-}_{l,\rho}\setminus A^{-}_{k,\rho}}
|\nabla u|\,\zeta^{c_{3}}dx\leqslant
K_{1}\frac{e^{c_{4}\lambda(\rho)}}{\varepsilon}\, \frac{M_{-}(l,\rho)}{\rho}\,
|A^{-}_{l,\rho}\setminus A^{-}_{k,\rho}|
\\
+  \frac{K_{1}\varepsilon^{c_{5}}\sigma^{-c_{6}}}{g\left(x_{0}, \frac{M_{-}(l,\rho)}{\rho} \right)}
\int\limits_{A^{-}_{l,\rho}}
g\left( x, \frac{K_{1}(u-l)_{-}}{\sigma\rho\,\zeta} \right)
\frac{(u-l)_{-}}{\rho}\, \zeta^{c_{3}-1}dx,
\end{multline}
provided that $M_{+}(k,\rho)\geqslant\rho$, $M_{-}(l,\rho)\geqslant\rho$.
Here we assume that $g$ satisfies condition 
(${\rm g}_{2}$) with
$K=2M$,  $(u-k)_{\pm}:=\max\{\pm(u-k),\, 0\}$,
$M_{\pm}(k,\rho):={\rm ess}\!\!\!\sup\limits_{B_{\rho}(x_{0})} (u-k)_{\pm}$,
$A^{\pm}_{k,\rho}:=B_{\rho}(x_{0})\cap \{(u-k)_{\pm}>0\}$.
}
\end{definition}

We refer to the parameters $M$, $K_{1}$, $n$, $q$,
$c_{1}$, $c_{2}$, $c_{3}$, $c_{4}$, $c_{5}$, $c_{6}$
as our structural data, and we write $\gamma$ if it can be
quantitatively determined a priori only in terms of the above
quantities. The generic constant $\gamma$ may change from line to line.

Our first main result of this section reads as follows.

\begin{theorem}\label{elliptth2.1}
Let $u\in\mathcal{B}_{1,g,\lambda}(\Omega)$, and let
$\rho_{0}$ be a sufficiently small positive number such that
$B_{8\rho_{0}}(x_{0})\subset \Omega$.
There exist positive numbers $c$, $\beta$ depending only on the data
such that if
\begin{equation}\label{Lambd32cond}
\exp\big( c\Lambda(\beta,r) \big)\leqslant
\frac{3}{2}\exp\big( c\Lambda(\beta,2r) \big),
\quad \Lambda(\beta,r):=\exp(\beta\lambda(r)),
\end{equation}
for all $0<r\leqslant \rho/2<\rho_{0}/2$, then
\begin{multline}\label{estellipth2.1}
{\rm osc}\{u; B_{r}(x_{0})\}:=
{\rm ess}\sup\limits_{B_{r}(x_{0})}u - {\rm ess}\inf\limits_{B_{r}(x_{0})}u
\\
\leqslant
2M\exp\left( -\gamma\int\limits_{2r}^{\rho}
\exp \big(-c\Lambda(\beta,t)\big)\frac{dt}{t} \right)
+\gamma(1+s_{0})\rho\exp \big( c\Lambda(\beta,\rho) \big),
\end{multline}
for all $0<2r<\rho<\rho_{0}$.

If additionally
\begin{equation}\label{condonLambdabetar}
\int\limits_{0} \exp \left(-c\Lambda(\beta,r)\right)\frac{dr}{r}=+\infty
\quad \text{and} \quad
\lim\limits_{r\rightarrow0}\, r \exp \left( c\Lambda(\beta,r) \right)=0,
\end{equation}
then $u$ is continuous at $x_{0}$.
\end{theorem}

We note that the function $\lambda(r)=L\ln\ln\ln \dfrac{1}{r}$
satisfies conditions \eqref{Lambd32cond}, \eqref{condonLambdabetar} if
$0<L<1/\beta$.
%

Theorem \ref{elliptth2.1} is an immediate consequence of the following two lemmas.

\begin{lemma}[De\,Giorgi type lemma]\label{DGellem2.3}
Let $u\in\mathcal{B}_{1,g,\lambda}(\Omega)$, and let
$B_{r}(x_{0})\subset \Omega$. Set
$$
\mu^{+}_{r}:={\rm ess}\sup\limits_{B_{r}(x_{0})}u, \quad
\mu^{-}_{r}:={\rm ess}\inf\limits_{B_{r}(x_{0})}u, \quad
\omega_{r}:=\mu^{+}_{r}-\mu^{-}_{r},
$$
and fix $\xi$, $a\in(0,1)$. Then there exists $\nu_{1}\in(0,1)$
depending only on the data and $a$
such that if
\begin{equation}
|\{x\in B_{r}(x_{0}): u(x)< \mu_{r}^{-}+\xi\,\omega_{r}\}|
\leqslant \nu_{1}e^{-c_{4}n\lambda(r)} |B_{r}(x_{0})|,
\end{equation}
then either
\begin{equation}\label{eqSkrel2.6}
a\,\xi\,\omega_{r} \leqslant (1+s_{0})r,
\end{equation}
or
\begin{equation*}
u(x)\geqslant \mu_{r}^{-}+a\,\xi\,\omega_{r} \quad \text{for a.a.} \ x\in B_{r/2}(x_{0}).
\end{equation*}
Likewise, if
\begin{equation}
|\{x\in B_{r}(x_{0}): u(x)> \mu_{r}^{+}-\xi\,\omega_{r}\}|
\leqslant \nu_{1}e^{-c_{4}n\lambda(r)} |B_{r}(x_{0})|,
\end{equation}
then either \eqref{eqSkrel2.6} holds true, or
\begin{equation*}
u(x)\leqslant \mu_{r}^{+}-a\,\xi\,\omega_{r} \quad \text{for a.a.} \ x\in B_{r/2}(x_{0}).
\end{equation*}
\end{lemma}
\begin{proof}
The proof of the lemma is similar to that of
\cite[Chap.\,II, Lemma\,6.1]{LadUr}.
%
For $j=0,1,2,\ldots$, we set $r_{j}:=\dfrac{r}{2}(1+2^{-j})$,
$k_{j}:=\mu^{-}_{r}+a\,\xi\,\omega_{r} + \dfrac{(1-a)\xi\,\omega_{r}}{2^{j}}$,
$B_{j}:=B_{r_{j}}(x_{0})$, $y_{j}:=|A^{-}_{k_{j}, r_{j}}|$.
And let $\zeta_{j}\in C_{0}^{\infty}(B_{j})$,
$0\leqslant\zeta_{j}\leqslant1$, $\zeta_{j}=1$ in $B_{j+1}$,
$|\nabla \zeta_{j}|\leqslant 2^{j+2}/r$.

Applying the Sobolev embedding theorem to the function
$(\max\{u, k_{j+1}\}-k_{j})_{-}\,\zeta_{j}^{c_{3}}$, 
we obtain
\begin{multline*}
\frac{(1-a)\xi\,\omega_{r}}{2^{j+1}}\, |A^{-}_{k_{j+1}, r_{j+1}}|^{1-\frac{1}{n}}=
(k_{j}-k_{j+1})\, |A^{-}_{k_{j+1}, r_{j+1}}|^{1-\frac{1}{n}}
\\ \leqslant
\bigg( \int\limits_{A^{-}_{k_{j+1}, r_{j+1}}}
(\max\{u, k_{j+1}\}-k_{j})_{-}^{\frac{n}{n-1}}\,dx \bigg)^{1-\frac{1}{n}}
\leqslant
\int\limits_{\Omega} \left|\nabla \big[ (\max\{u, k_{j+1}\}-k_{j})_{-}\,\zeta_{j}^{c_{3}}\big]\right| dx
\\
\leqslant \int\limits_{A^{-}_{k_{j}, r_{j}}\setminus A^{-}_{k_{j+1}, r_{j}}}
|\nabla u|\,\zeta_{j}^{c_{3}} dx+c_{3}
\frac{2 (1-a) \xi\,\omega_{r}}{r}\,|A^{-}_{k_{j}, r_{j}}|. 
\end{multline*}
Next, using inequality \eqref{eqSkrel2.2} with
$l=k_{j}$, $k=k_{j+1}$, $\varepsilon=1$, $\sigma=(2^{j+1}+2)^{-1}$,
$\rho=r_{j}$, $\zeta=\zeta_{j}$,
we obtain that
\begin{multline*}
\int\limits_{A^{-}_{k_{j}, r_{j}}\setminus A^{-}_{k_{j+1}, r_{j}}}
|\nabla u|\,\zeta_{j}^{c_{3}} dx
\leqslant
2K_{1} e^{c_{4}\lambda(r)}\, \frac{\xi\,\omega_{r}}{r}
\,|A^{-}_{k_{j}, r_{j}}|
\\
+\frac{K_{1}2^{c_{6}(j+2)+1} }
{ g\left( x_{0}, \frac{M_{-}(k_{j},r_{j})}{r} \right)}
\frac{\xi\,\omega_{r}}{r} \int\limits_{A^{-}_{k_{j},r_{j}}}
g\left( x, \frac{K_{1}2^{j+2}M_{-}(k_{j},r_{j})}{\zeta_{j}\,r} \right)\,
\zeta_{j}^{c_{3}-1}dx.
\end{multline*}
If inequality \eqref{eqSkrel2.6} is violated, then
$$
\frac{K_{1}2^{j+2}M_{-}(k_{j},r_{j})}{\zeta_{j}\,r}>
\frac{M_{-}(k_{j},r_{j})}{r}\geqslant \frac{a\,\xi\,\omega_{r}}{r}>1+s_{0},
$$
and by (${\rm g}_{1}$) and (${\rm g}_{2}$),
the following inequality holds for $x\in A^{-}_{k_{j},r_{j}}$:
\begin{multline*}
g\left( x, \frac{K_{1}2^{j+2}M_{-}(k_{j},r_{j})}{\zeta_{j}\,r} \right)\,
\leqslant
c_{1}
\left( \frac{K_{1}2^{j+2}}{\zeta_{j}} \right)^{q-1}
g\left( x, \frac{M_{-}(k_{j},r_{j})}{r} \right)
\\
\leqslant
\gamma\, 2^{(j+2)(q-1)}
e^{\lambda(r)}g\left( x_{0}, \frac{M_{-}(k_{j},r_{j})}{r} \right)
\zeta_{j}^{1-q}
\end{multline*}
From the previous we arrive at 
\begin{equation*}\label{yj+1yj}
y_{j+1}\leqslant
\frac{\gamma\, 2^{j\gamma}}{1-a}  \, \frac{e^{c_{4}\lambda(r)}}{r}\,
y_{j}^{1+\frac{1}{n}}, \quad j=0,1,2,\ldots
\end{equation*}
From this, by Lemma \ref{lemonyj}, it follows that
$\lim\limits_{j\rightarrow\infty}y_{j}=0$,
provided $\nu_{1}$ is chosen to satisfy
$$
\nu_{1}:=\gamma^{-1}(1-a)^{n},
$$
which proves the lemma.
\end{proof}

\begin{lemma}[Expansion of positivity]\label{DGellem2.4}
Let $u\in\mathcal{B}_{1,g,\lambda}(\Omega)$, and let $B_{8r}(x_{0})\subset \Omega$
and $\xi\in(0,1)$. Assume that with some $\alpha\in (0,1)$
there holds
\begin{equation}\label{eqSkrel2.10}
|\{x\in B_{r}(x_{0}): u(x)< \mu^{-}_{r}+\xi\,\omega_{r}\}|
\leqslant (1-\alpha)\, |B_{r}(x_{0})|.
\end{equation}
Then either 
\begin{equation}\label{eqSkrel2.11}
\omega_{r}\leqslant 2^{j_{\ast}}(1+s_{0})\,r
\end{equation}
or
\begin{equation}\label{alteqSkrel2.11}
u(x)\geqslant \mu_{r}^{-}+\omega_{r}/2^{j_{\ast}+1} \ \
\text{for a.a.} \ x\in B_{r/2}(x_{0}),
\end{equation}
where
\begin{equation}\label{defeljast}
j_{\ast}= 2+\log_{2}\xi^{-1}+(\gamma/\nu_{1})^{\frac{(n-1)(c_{5}+1)}{nc_{5}}}
e^{\beta\lambda(r/4)}, \quad
\beta=\frac{2c_{4}(n-1)(c_{5}+1)}{nc_{5}},
\end{equation}
and $\nu_{1}$ is the constant from the previous lemma.

Likewise, if
\begin{equation}
|\{x\in B_{r}(x_{0}): u(x)> \mu_{r}^{+}-\xi\,\omega_{r}\}|
\leqslant (1-\alpha)\, |B_{r}(x_{0})|,
\end{equation}
then either \eqref{eqSkrel2.11} holds true, or
\begin{equation}
u(x)\leqslant \mu_{r}^{+}-\omega_{r}/2^{j_{\ast}+1} \ \
\text{for a.a.} \ x\in B_{r/2}(x_{0}).
\end{equation}
\end{lemma}
\begin{proof}
The proof of the lemma is completely similar to that of
\cite[Chap.\,II, Lemma\,6.2]{LadUr}.

Without loss of generality we assume that
$\log_{2}\xi^{-1}\in \mathbb{N}$.
We set
$k_{j}:=\mu_{r}^{-}+\omega_{r}/2^{j}$ for every
$j=\log_{2}\xi^{-1},\, 1+\log_{2}\xi^{-1},\, 2+\log_{2}\xi^{-1},\ldots,\, j_{\ast}-1$.
Let $\zeta\in C_{0}^{\infty}(B_{r}(x_{0}))$, $0\leqslant\zeta\leqslant1$,
$\zeta=1$ in $B_{r/2}(x_{0})$, $|\nabla \zeta|\leqslant2/r$.
For fixed $j$, inequality \eqref{eqSkrel2.2} with
$l=k_{j}$, $k=k_{j+1}$, $\rho=r$, $\sigma=1/2$ and
$$
\varepsilon=\left( \frac{|A^{-}_{k_{j},r}\setminus A^{-}_{k_{j+1},r}|}
{|B_{r}(x_{0})|} \right)^{\frac{1}{1+c_{5}}}
$$
yields
\begin{multline}\label{eqSkrel2.2ellem2.4}
\int\limits_{A^{-}_{k_{j},r}\setminus A^{-}_{k_{j+1},r}}
|\nabla u|\,\zeta^{c_{3}} dx \leqslant
\frac{M_{-}(k_{j},r)}{r}
\left(\frac{|A^{-}_{k_{j},r}\setminus
A^{-}_{k_{j+1},r}|}{|B_{r}(x_{0})|}\right)^{\frac{c_{5}}{c_{5}+1}}
|B_{r}(x_{0})|\,e^{c_{4}\lambda(r)}
\\
+\frac{M_{-}(k_{j},r)}{r}
\left(\frac{|A^{-}_{k_{j},r}\setminus
A^{-}_{k_{j+1},r}|}{|B_{r}(x_{0})|}\right)^{\frac{c_{5}}{c_{5}+1}}
\frac{2^{c_{6}}K_{1}}{g\left(x_{0}, \frac{M_{-}(k_{j},r)}{r}\right)}
\int\limits_{A^{-}_{k_{j},r}} g\left( x, \frac{2K_{1}M_{-}(k_{j},r)}{r\,\zeta} \right)
\zeta^{c_{3}-1}\,dx.
\end{multline}
If \eqref{eqSkrel2.11} is violated, then
$$
\frac{2K_{1}M_{-}(k_{j},r)}{r\,\zeta}>
\frac{M_{-}(k_{j},r)}{r}=
\frac{\omega_{r}}{2^{j}r}\geqslant
\frac{\omega_{r}}{2^{j_{\ast}}r}
>1+s_{0}.
$$
Therefore, we can apply conditions (${\rm g}_{1}$) and (${\rm g}_{2}$) to estimate the integral on
the right-hand side of \eqref{eqSkrel2.2ellem2.4} as follows:
\begin{multline*}
\int\limits_{A^{-}_{k_{j},r}} g\left( x, \frac{2K_{1}M_{-}(k_{j},r)}{r\,\zeta} \right)
\zeta^{c_{3}-1}\,dx
\leqslant c_{1}(2K_{1})^{q-1}
\int\limits_{A^{-}_{k_{j},r}} g\left( x, \frac{M_{-}(k_{j},r)}{r} \right)
\zeta^{c_{3}-q}\,dx
\\
\leqslant
c_{1}c_{2}(M) (2K_{1})^{q-1}\,e^{\lambda(r)}\,|B_{r}(x_{0})|\,
g\Big( x_{0}, \frac{M_{-}(k_{j},r)}{r} \Big).
\end{multline*}
The combination of this inequality and \eqref{eqSkrel2.2ellem2.4} gives the estimate
\begin{equation}\label{LadUr6.22}
\int\limits_{A^{-}_{k_{j},r}\setminus A^{-}_{k_{j+1},r}}
|\nabla u|\,\zeta^{c_{3}} dx
\leqslant
\gamma\,\frac{M_{-}(k_{j},r)}{r}
\left(\frac{|A^{-}_{k_{j},r}\setminus
A^{-}_{k_{j+1},r}|}{|B_{r}(x_{0})|}\right)^{\frac{c_{5}}{c_{5}+1}}
|B_{r}(x_{0})|\,e^{c_{4}\lambda(r)}.
\end{equation}

On the other hand, inequality \eqref{eq2.2} with $l=k_{j}$, $k=k_{j+1}$,
$\rho=r/2$, and condition \eqref{eqSkrel2.10} imply that for $j\leqslant j_{\ast}$,
\begin{multline}\label{LadUr6.23}
|A^{-}_{k_{j_{\ast}},r}|^{1-\frac{1}{n}}
\leqslant \frac{c}{\alpha |B_{1}(x_{0})|}\, \frac{2^{j+1}}{\omega_{r}}
\int\limits_{A^{-}_{k_{j},r/2}\setminus A^{-}_{k_{j+1},r/2}}
|\nabla u|\,dx
=
\frac{\gamma}{M_{-}(k_{j},r)}
\int\limits_{A^{-}_{k_{j},r/2}\setminus A^{-}_{k_{j+1},r/2}}
|\nabla u|\,dx.
\end{multline}
Combining \eqref{LadUr6.22} and \eqref{LadUr6.23}, we get
$$
|A^{-}_{k_{j_{\ast}},r}|^{1-\frac{1}{n}}
\leqslant
\gamma\, \frac{e^{c_{4}\lambda(r)}}{r}
\left(\frac{|A^{-}_{k_{j},r}\setminus
A^{-}_{k_{j+1},r}|}{|B_{r}(x_{0})|}\right)^{\frac{c_{5}}{c_{5}+1}}
|B_{r}(x_{0})|, \quad j\leqslant j_{\ast},
$$
which implies
$$
|A^{-}_{k_{j_{\ast}},r}|^{\frac{(c_{5}+1)(n-1)}{c_{5}n}}
\leqslant
\gamma \bigg( \frac{e^{c_{4}\lambda(r)}}{r} \bigg)^{\frac{c_{5}+1}{c_{5}}}
|B_{r}(x_{0})|^{\frac{1}{c_{5}}}\,
|A^{-}_{k_{j},r}\setminus A^{-}_{k_{j+1},r}|.
$$
Summing up this inequality for $j=\log_{2}\xi^{-1},\ldots, j_{\ast}$,
we obtain 
\begin{equation}\label{estAkjastel}
|A^{-}_{k_{j_{\ast}},r}|
\leqslant
\gamma (j_{\ast}-\log_{2}\xi^{-1}-1)^{-\frac{nc_{5}}{(n-1)(c_{5}+1)}}
\exp\left(\frac{c_{4}n\,\lambda(r)}{n-1}\right) |B_{r}(x_{0})|.
\end{equation}
From this using \eqref{defeljast} and Lemma \ref{DGellem2.3},
we arrive at \eqref{alteqSkrel2.11}, which complete the proof of the lemma.
%
\end{proof}

\emph{\textbf{Proof of Theorem \ref{elliptth2.1}.}}
Fix $\rho\leqslant\rho_{0}$. The following two alternative
cases are possible
$$
|\{x\in B_{\rho}(x_{0}):u(x)>\mu^{+}_{\rho}-\omega_{\rho}/2\}|
\leqslant\frac{1}{2}|B_{\rho}(x_{0})|
$$
or
$$
|\{x\in B_{\rho}(x_{0}):u(x)<\mu^{-}_{\rho}+\omega_{\rho}/2\}|
\leqslant\frac{1}{2}|B_{\rho}(x_{0})|.
$$
Assume, for example, the first one.
Then, by Lemma \ref{DGellem2.4}, we obtain
$$
\omega_{\rho/2}\leqslant
\left[ 1-\exp\big(-\gamma\Lambda(\beta,\rho)\big) \right]\omega_{\rho}
+
\gamma(1+s_{0})\rho\exp\big(\gamma\Lambda(\beta,\rho)\big).
$$
Iterating this inequality 
we have for any $j\geqslant1$,
\begin{multline*}
\omega_{\rho_{j}}\leqslant \omega_{\rho}
\prod_{i=0}^{j-1}
\left[1-\exp \big(-\gamma\Lambda(\beta,\rho_{i})  \big) \right]
+
\gamma(1+s_{0})\rho_{j-1}\exp \big(\gamma\Lambda(\beta,\rho_{j-1})  \big)
\\
+\gamma(1+s_{0}) \sum_{i=0}^{j-2}
\rho_{i}\exp \big(\gamma\Lambda(\beta,\rho_{i})  \big)
\prod_{k=i+1}^{j-1}
\left[1-\exp \big(-\gamma\Lambda(\beta,\rho_{k})  \big) \right],
 \ \ \rho_{j}:=2^{-j}\rho.
\end{multline*}
Hence, using \eqref{Lambd32cond} and the fact that
$$
\prod_{i=0}^{j-1}
\left[1-\exp \big(-\gamma\Lambda(\beta,\rho_{i})  \big) \right]
\leqslant
\exp \left(-\sum_{i=0}^{j-1} \exp \big(-\gamma\Lambda(\beta,\rho_{i})  \big) \right),
$$
we obtain
$$
\omega_{\rho_{j}}\leqslant \omega_{\rho}
\exp \left(-\sum_{i=0}^{j-1} \exp \big(-\gamma\Lambda(\beta,\rho_{i})  \big) \right)
+\gamma(1+s_{0})\rho\exp\big(\gamma\Lambda(\beta,\rho)\big),
$$
from which the required \eqref{estellipth2.1} follows.
This completes the proof of Theorem \ref{elliptth2.1}.


\subsection{Elliptic equations with generalized Orlicz growth}

We consider the equation
\begin{equation}\label{gellequation}
{\rm div}\bigg( g(x,|\nabla u|)\frac{\nabla u}{|\nabla u|} \bigg)=0,
\quad x\in \Omega.
\end{equation}
Additionally, we assume that
\begin{itemize}
\item[(${\rm g}_{3}$)]
there exist $c_{7}>0$, $p\in(1,q)$ such that,
for a.a. $x\in \Omega$ and for ${\rm w}\geqslant{\rm v}>0$,
$$
\frac{g(x, {\rm w})}{g(x, {\rm v})}\geqslant
c_{7} \left( \frac{{\rm w}}{{\rm v}} \right)^{p-1}.
$$
\end{itemize}

We set ${\rm G}(x, {\rm v}):=g(x, {\rm v}){\rm v}$ for ${\rm v}>0$, and
write $W^{1,{\rm G}}(\Omega)$ for the class of functions which are weakly
differentiable in $\Omega$ with
$
\int\limits_{\Omega} {\rm G}(x, |\nabla u|)\,dx<+\infty.
$

By definition, a function $u\in W^{1,{\rm G}}(\Omega)\cap L^{\infty}(\Omega)$ is a
solution of Eq. \eqref{gellequation} if, for any function
$\varphi\in W^{1,{\rm G}}_{0}(\Omega):=
\left\{ w\in W^{1,{\rm G}}(\Omega): w \text{ has a compact support in } \Omega\right\}$,
the following identity is true:
\begin{equation}\label{gelintidentity}
\int\limits_{\Omega}
g(x,|\nabla u|)\,\frac{\nabla u}{|\nabla u|}\,\nabla\varphi\,dx=0.
\end{equation}

We show that the solutions of Eq. \eqref{gellequation} belong to the corresponding
$\mathcal{B}_{1}$ classes.

First, note a simple analogues of Young's inequality:
\begin{equation}\label{gYoungineq1}
g(x,a)b\leqslant \varepsilon g(x,a)a+g(x,b/\varepsilon)b,
\end{equation}
\begin{equation}\label{gYoungineq2}
g(x,a)b\leqslant \frac{1}{\varepsilon}\,g(x,a)a
+ \frac{\varepsilon^{p-1}}{ c_{7}}\,g(x,b)b,
\end{equation}
which are valid for any $\varepsilon\in (0,1)$, $a,b>0$ and for a.a. $x\in \Omega$.

Indeed, if $b\leqslant \varepsilon a$, then $g(x,a)b\leqslant \varepsilon g(x,a)a$,
and if $b>\varepsilon a$, then since the function ${\rm v}\rightarrow g(\cdot, {\rm v})$
is nondecreasing we have that $g(x,a)b\leqslant g(x,b/\varepsilon)b$,
that proves \eqref{gYoungineq1}. Using assumption (${\rm g}_{3}$) and similar arguments,
we arrive at \eqref{gYoungineq2}.

Let $u\in W^{1,{\rm G}}(\Omega)\cap L^{\infty}(\Omega)$ be a solution of Eq. \eqref{gellequation}.
We fix concentric balls
$B_{(1-\sigma)\rho}(x_{0})\subset B_{\rho}(x_{0})\subset\Omega$ and a function
$\zeta\in C_{0}^{\infty}(B_{\rho}(x_{0}))$, $0\leqslant\zeta\leqslant1$,
$\zeta=1$ in $B_{(1-\sigma)\rho}(x_{0})$, $|\nabla\zeta|\leqslant (\sigma\rho)^{-1}$.
Testing \eqref{gelintidentity} by $\varphi=(u-k)_{+}\zeta^{c_{3}}$, we obtain
$$
\int\limits_{A^{+}_{k,\rho}}
g(x,|\nabla u|)\, |\nabla u| \,\zeta^{c_{3}}dx\leqslant
c_{3}\int\limits_{A^{+}_{k,\rho}}
g(x,|\nabla u|)\, \frac{(u-k)_{+}}{\sigma\rho}\,\zeta^{c_{3}-1}dx.
$$
From this by \eqref{gYoungineq1} we get
$$
\int\limits_{A^{+}_{k,\rho}}
g(x,|\nabla u|)\, |\nabla u| \,\zeta^{c_{3}}dx\leqslant
\frac{\gamma(c_{3})}{\sigma\rho}
\int\limits_{A^{+}_{k,\rho}}
g\left( x, \frac{\gamma(c_{3}) (u-k)_{+}}{\sigma\rho\,\zeta} \right)
(u-k)_{+}\, \zeta^{c_{3}-1}dx
$$
If $M_{+}(k,\rho)\geqslant\rho$, then using \eqref{gYoungineq2} and condition (${\rm g}_{3}$), we have for any
$\varepsilon_{1}\in(0,1)$,
\begin{multline*}
\int\limits_{A^{+}_{k,\rho}\setminus A^{+}_{l,\rho}}
g\left(x_{0}, \frac{M_{+}(k,\rho)}{\rho}  \right)
|\nabla u|\,\zeta^{c_{3}}dx
\leqslant
c_{2}(M)e^{\lambda(\rho)}\int\limits_{A^{+}_{k,\rho}\setminus A^{+}_{l,\rho}}
g\left(x, \frac{M_{+}(k,\rho)}{\rho}  \right)
|\nabla u|\,\zeta^{c_{3}}dx
\\
\leqslant \frac{c_{2}(M)e^{\lambda(\rho)}}{\varepsilon_{1}}
\frac{M_{+}(k,\rho)}{\rho}
\int\limits_{A^{+}_{k,\rho}\setminus A^{+}_{l,\rho}}
g\left(x, \frac{M_{+}(k,\rho)}{\rho}  \right)dx+
c_{2}(M)c_{7}\varepsilon_{1}^{p-1}e^{\lambda(\rho)}
\int\limits_{A^{+}_{k,\rho}}
G(x, |\nabla u|)\,\zeta^{c_{3}}dx
\\
\leqslant
\frac{c_{2}^{2}(M)e^{2\lambda(\rho)}}{\varepsilon_{1}}
g\left(x_{0}, \frac{M_{+}(k,\rho)}{\rho}  \right)
\frac{M_{+}(k,\rho)}{\rho} |A^{+}_{k,\rho}\setminus A^{+}_{l,\rho}|
\\
+\gamma(c_{2}(M),c_{3},c_{7})\varepsilon_{1}^{p-1} \frac{e^{\lambda(\rho)}}{\sigma\rho}
\int\limits_{A^{+}_{k,\rho}}
g\left( x, \frac{\gamma(c_{3}) (u-k)_{+}}{\sigma\rho\,\zeta} \right)
(u-k)_{+}\, \zeta^{c_{3}-1}dx.
\end{multline*}
Choosing $\varepsilon_{1}$ from the condition
$\varepsilon_{1}=\varepsilon e^{-\frac{\lambda(\rho)}{p-1}}$,
we arrive at \eqref{eqSkrel2.1}.
The proof of \eqref{eqSkrel2.2} is completely similar.

Taking into account the previous arguments, we arrive at
\begin{theorem}\label{contgsolution}
Let $u\in W^{1,{\rm G}}(\Omega)\cap L^{\infty}(\Omega)$ be a solution
of Eq.  \eqref{gellequation} under assumptions {\rm (${\rm g}_{1}$)}, {\rm (${\rm g}_{2}$)},
{\rm (${\rm g}_{3}$)}, \eqref{Lambd32cond} and \eqref{condonLambdabetar},
then $u\in C_{{\rm loc}}(\Omega)$.
\end{theorem}
\begin{remark}
{\rm We note that in the case, when
$g(x, {\rm v}):={\rm v}^{p(x)-1}$ (${\rm v}>0$, $x\in\Omega$),
$$
\text{osc}\{p(x);B_{r}(x_{0})\}\leqslant \frac{\lambda(r)}{\ln\frac{1}{r}}
$$
and $\lambda(r)$ satisfies \eqref{condonLambdabetar},
Theorem \ref{contgsolution} was proved in \cite{AlkhSurnAlgAn19}.
 We also note that in the case
$$
0\leqslant \lambda(r)\leqslant L<+\infty,
$$
conditions (${\rm g}_{1}$), (${\rm g}_{2}$), (${\rm g}_{3}$)
are almost equivalent to the conditions $({\rm aDec})_{q}^{\infty}$,
(${\rm A}1$-n), $({\rm aInc})_{p}$  in \cite{HarHastLee18}.}
\end{remark}

%
%

\section{Parabolic $\mathcal{B}_{1,g,\lambda}$ classes}


We consider a function $g:\mathbb{R}^{n}\times \mathbb{R}_{+} \times \mathbb{R}_{+}\rightarrow \mathbb{R}_{+}$
having the following properties:
for every ${\rm v}\in \mathbb{R}_{+}$ the function $g(\cdot,\cdot,{\rm v})$
is measurable on $\mathbb{R}^{n}\times \mathbb{R}_{+}$ and,
for almost every $(x,t)\in \mathbb{R}^{n}\times \mathbb{R}_{+}$,
the function $g(x,t,\cdot)$ is increasing and continuous in $\mathbb{R}_{+}$,
and
$$
\lim\limits_{{\rm v}\rightarrow +0}g(x,t,{\rm v})=0, \quad
\lim\limits_{{\rm v}\rightarrow +\infty}g(x,t,{\rm v})=+\infty.
$$


In addition, let $\lambda(r)$ be a continuous and non-increasing
function on the interval $(0,1)$.

Let $(x_{0},t_{0})\in \mathbb{R}^{n}\times \mathbb{R}_{+}$ and construct
the cylinder
$$
Q_{R,R}(x_{0},t_{0}):=
B_{R}(x_{0})\times (t_{0}-R,t_{0})\subset \mathbb{R}^{n}\times \mathbb{R}_{+}.
$$

\begin{definition}\label{defB1class}
{\rm 
We say that a measurable function $u: Q_{R,R}(x_{0},t_{0})\rightarrow \mathbb{R}$ belongs to the
parabolic class $\mathcal{B}_{1,g,\lambda}(Q_{R,R}(x_{0},t_{0}))$, if
$$
u\in C\big(t_{0}-R,t_{0}; L^{2}(B_{R}(x_{0}))\big)\cap
L^{1}\big( t_{0}-R,t_{0}; W^{1,1}(B_{R}(x_{0}))\big)
\cap L^{\infty}(Q_{R,R}(x_{0},t_{0})),
$$
${\rm ess}\!\!\!\!\!\!\sup\limits_{Q_{R,R}(x_{0},t_{0})}\!\!\!\!\!\!|u|\leqslant M$,
and there exist positive numbers $K_{1}$,
$c_{1}$, $c_{2}$, $c_{3}$, $c_{4}$ such that for 
any cylinder
$$Q_{8r,8\theta}(x_{0}, t_{0}):=B_{8r}(x_{0})\times (t_{0}-8\theta, t_{0})
\subset Q_{8r,8r}(x_{0},t_{0}) \subset Q_{R,R}(x_{0},t_{0}),
$$
any $k,l\in \mathbb{R}$,
$k<l$, $|k|,|l|<M$, any $\varepsilon\in(0,1]$, any $\sigma\in(0,1)$,
for any $\zeta(x)\in C_{0}^{\infty}(B_{r}(x_{0}))$,
$0\leqslant\zeta(x)\leqslant1$, $\zeta(x)=1$ in $B_{r(1-\sigma)}(x_{0})$,
$|\nabla \zeta|\leqslant (r\sigma)^{-1}$,
and for any
$\chi(t)\in C^{1}(\mathbb{R}_{+})$, $0\leqslant \chi(t)\leqslant1$, the following inequalities hold:
\begin{multline}\label{eqSkrnew2.2}
\iint\limits_{A^{+}_{k,r,\theta}\setminus A^{+}_{l,r,\theta}} |\nabla u|\,\zeta^{c_{1}}\chi \,dxdt\leqslant
K_{1}\frac{e^{c_{2}\lambda(r)}}{\varepsilon}\,\frac{M_{+}(k,r,\theta)}{r}
|A^{+}_{k,r,\theta}\setminus A^{+}_{l,r,\theta}|
\\
+\frac{K_{1}\sigma^{-c_{3}}\varepsilon^{c_{4}}}{g\left(x_{0},t_{0},\frac{M_{+}(k,r,\theta)}{r}\right)}
\bigg\{ \int\limits_{B_{r}(x_{0})\times \{t_{0}-\theta\}} (u-k)_{+}^{2}\zeta^{c_{1}}\chi(t_{0}-\theta)dx +
\iint\limits_{A^{+}_{k,r,\theta}}(u-k)_{+}^{2}\,|\chi_{t}|\,\zeta^{c_{1}} dxdt
\\
+  \iint\limits_{A^{+}_{k,r,\theta}} g\bigg(x,t,\frac{K_{1}(u-k)_{+}}
{\sigma r\zeta}\bigg) \frac{(u-k)_{+}}{r} \,\chi\,\zeta^{c_{1}-1}dxdt \bigg\},
\end{multline}
\begin{multline}\label{eqSkrnew2.3}
\iint\limits_{A^{-}_{l,r,\theta}\setminus A^{-}_{k,r,\theta}}
|\nabla u|\,\zeta^{c_{1}}\chi \,dxdt\leqslant
K_{1}\frac{e^{c_{2}\lambda(r)}}{\varepsilon}\,\frac{M_{-}(l,r,\theta)}{r}
|A^{-}_{l,r,\theta}\setminus A^{-}_{k,r,\theta}|
\\
+\frac{K_{1}\sigma^{-c_{3}}\varepsilon^{c_{4}}}{g\left(x_{0},t_{0},\frac{M_{-}(l,r,\theta)}{r}\right)}
\bigg\{ \int\limits_{B_{r}(x_{0})\times \{t_{0}-\theta\}} (u-l)_{-}^{2}\zeta^{c_{1}}\chi(t_{0}-\theta)dx +
\iint\limits_{A^{-}_{l,r,\theta}}(u-l)_{-}^{2}\,|\chi_{t}|\,\zeta^{c_{1}} dxdt
\\
+ \iint\limits_{A^{-}_{l,r,\theta}} g\bigg(x,t,\frac{K_{1}(u-l)_{-}}{\sigma r\zeta}\bigg)
 \frac{(u-l)_{-}}{r} \,\chi\,\zeta^{c_{1}-1}dxdt \bigg\},
\end{multline}
provided that $M_{+}(k,r,\theta)\geqslant r$, $M_{-}(l,r,\theta)\geqslant r$,
\begin{multline}\label{eqSkr2.4}
\int\limits_{B_{r}(x_{0})\times\{t\}} (u-k)_{\pm}^{2}\zeta^{c_{1}}\chi dx\leqslant
\int\limits_{B_{r}(x_{0})\times\{t_{0}-\theta\}} (u-k)_{\pm}^{2}\zeta^{c_{1}}\chi(t_{0}-\theta)dx
\\
+\frac{K_{1}}{\sigma^{c_{3}}} \bigg\{ \iint\limits_{A^{\pm}_{k,r,\theta}} (u-k)_{\pm}\, |\chi_{t}|\, \zeta^{c_{1}} dxdt +
\iint\limits_{A^{\pm}_{k,r,\theta}}
g\bigg( x,t, \frac{K_{1}(u-k)_{\pm}}{\sigma r\zeta} \bigg) \frac{(u-k)_{\pm}}{r}\,
 \chi\,\zeta^{c_{1}-1} dxdt\bigg\},
\end{multline}
for all $t\in (t_{0}-\theta,t_{0})$, where
$(u-k)_{\pm}:=\max\{\pm(u-k),\, 0\}$,
$$
M_{\pm}(k,r,\theta):={\rm ess}\!\!\!\!\!\!
\sup\limits_{Q_{r,\theta}(x_{0},t_{0})}\!\!\!\!\!\!
(u-k)_{\pm}, \quad
A^{\pm}_{k,r,\theta}:=Q_{r,\theta}(x_{0},t_{0})\cap \{(u-k)_{\pm}>0\}.
$$
}
\end{definition}

Further, we also assume that
\begin{itemize}

\item[(${\rm g}$)] there exists a positive constant $c_{5}$ such that,
for all $(x,t)\in Q_{R,R}(x_{0},t_{0})$,
\begin{equation}\label{eqSkrnew2.6}
c_{5}^{-1}\leqslant g(x,t,1)\leqslant c_{5}.
\end{equation}

\item[(${\rm g}_{1}$)]
there exist positive constants
$K$, $K_{2}$, $c_{6}$ such that, for any
$(x_{1}, t_{1})$, $(x_{2}, t_{2})\in Q_{r,Kr}(x_{0},t_{0})\subset Q_{R,R}(x_{0},t_{0})$
and for all ${\rm v}\in (r,M)$, there holds
\begin{equation}\label{eqSkrnew2.5}
g(x_{1}, t_{1},{\rm v}/r)\leqslant K_{2}e^{c_{6}\lambda(r)}g(x_{2}, t_{2}, {\rm v}/r);
\end{equation}

\end{itemize}


We will also suppose that the function
$$
\psi(x,t,{\rm v}):=\frac{g(x,t,{\rm v})}{{\rm v}},
\quad (x,t,{\rm v})\in \mathbb{R}^{n}\times \mathbb{R}_{+} \times \mathbb{R}_{+},
$$
satisfies one of the following conditions:
\begin{itemize}
\item[(${\rm g}_{2}\psi_{1}$)]
,,\textbf{degenerate}'' case:  with some $b_{0}$, $\delta\geqslant0$, $K$,
$c_{7}$, $c_{8}>0$, $0<\mu_{1}<\mu_{2}$ there hold
\begin{equation}\label{psib0cond}
\frac{\psi(x_{0},t_{0},{\rm w})}{\psi(x_{0},t_{0},{\rm v})}
\geqslant c_{7} \Big(\frac{{\rm w}}{{\rm v}}\Big)^{\mu_{1}}
\ \ \text{if} \ \ {\rm w}\geqslant {\rm v}>b_{0}R^{-\delta},
\hskip 43 mm
\end{equation}
\begin{equation}\label{eqSkr2.7}
\frac{\psi(x,t,{\rm w})}{\psi(x,t,{\rm v})}
\leqslant c_{8} \Big(\frac{{\rm w}}{{\rm v}}\Big)^{\mu_{2}}
\ \ \text{if} \ \ {\rm w}\geqslant {\rm v}>b_{0}R^{-\delta} \ \text{ and } \
(x,t)\in Q_{R,R}(x_{0},t_{0}),
\end{equation}
\end{itemize}
or
\begin{itemize}
\item[(${\rm g}_{2}\psi_{2}$)]
,,\textbf{singular}'' case:
with some $b_{0}$, $\delta\geqslant0$, $K$, $c_{9}$, $c_{10}>0$, $0<\mu_{3}<\mu_{4}<1$
there hold
\begin{equation}\label{eqeqSkr2.8mu4}
\frac{\psi(x_{0},t_{0},{\rm w})}{\psi(x_{0},t_{0},{\rm v})}
\geqslant c_{9} \Big(\frac{{\rm v}}{{\rm w}}\Big)^{\mu_{4}}
\ \ \text{if} \ \ {\rm w}\geqslant {\rm v}>b_{0}R^{-\delta},
\hskip 43mm
\end{equation}
\begin{equation}\label{eqeqSkr2.8mu3}
\frac{\psi(x,t,{\rm w})}{\psi(x,t,{\rm v})}
\leqslant c_{10} \Big(\frac{{\rm v}}{{\rm w}}\Big)^{\mu_{3}}
\ \ \text{if} \ \ {\rm w}\geqslant {\rm v}>b_{0}R^{-\delta} \ \text{ and } \
(x,t)\in Q_{R,R}(x_{0},t_{0}).
\end{equation}
In this case we additionally assume that for all $t\in (t_{0}-\theta, t_{0})$ and
$\varepsilon M_{\pm}(k,r,\theta)\geqslant r$
the following inequality holds:
\begin{equation}\label{eqSkr2.9}
\begin{aligned}
&D^{-}\int\limits_{B_{r}(x_{0})\times \{t\}} \Phi_{k}(x_{0}, t_{0},u)\,\frac{t-t_{0}+\theta}{\theta}\, \zeta^{c_{1}}dx
\\
&+\frac{r}{K_{1}} \int\limits_{B_{r}(x_{0})\times \{t\}}
\bigg|\nabla \ln \frac{(1+\varepsilon)M_{\pm}(k,r,\theta)}{w_{k}}\bigg|\,
\frac{t-t_{0}+\theta}{\theta}\, \zeta^{c_{1}}dx
\\
&\leqslant \frac{K_{1}}{\theta}\int\limits_{B_{r}(x_{0})\times \{t\}} \Phi_{k}(x_{0}, t_{0},u)\, \zeta^{c_{1}}dx
+\frac{K_{1}}{\sigma^{c_{3}}r}\int\limits_{B_{r}(x_{0})\times \{t\}}
\frac{g\bigg(x,t, \dfrac{K_{1}w_{k}}{\sigma r\zeta}\bigg)w_{k}}{G(x_{0},t_{0},w_{k}/r)}\, \zeta^{c_{1}-1}dx
\\
&+K_{1}e^{c_{2}\lambda(r)}|B_{r}(x_{0})|,
\end{aligned}
\end{equation}
where
$G(x,t,z):=\int\limits_{0}^{z}g(x,t,s)ds$, \
$w_{k}:=(1+\varepsilon)M_{\pm}(k,r,\theta)-(u-k)_{\pm}$,
$$
\Phi_{k}(x,t,u):=\int\limits_{0}^{(u-k)_{\pm}}
\frac{(1+\varepsilon)M_{\pm}(k,r,\theta)-s}{G\bigg(x,t, \dfrac{(1+\varepsilon)M_{\pm}(k,r,\theta)-s}{r}\bigg)}\,ds,
$$
and the notation $D^{-}$ is used to denote the derivative
$$
D^{-}f(t):=\limsup\limits_{h\rightarrow0}\frac{f(t)-f(t-h)}{h}.
$$
\end{itemize}

The parameters $n$, $K$, $K_{1}$, $K_{2}$, $c$, $c_{i}$, $i=1,\ldots, 10$, $\mu_{1}$, $\mu_{2}$, $\mu_{3}$, $\mu_{4}$, $M$ are the data, and we say
that a generic constant $\gamma$ depends upon the data, if it can be quantitatively determined
a priory only in terms of the indicated parameters.

Our main results of this Section are read as follows.

\begin{theorem}\label{th2.1}
Let $u\in \mathcal{B}_{1,g,\lambda}(Q_{R,R}(x_{0},t_{0}))$ and
let hypotheses $({\rm g})$, $({\rm g}_{1})$, $({\rm g}_{2}\psi_{1})$ be fulfilled.
Then there exist positive numbers $c$, $\beta$ depending only on the data such that if
\begin{equation}\label{eqSkr3.112505}
\exp\big(c\Lambda(\beta,r)\big)\leqslant \left(\frac{3}{2}\right)^{1-\delta_{0}}
\exp\big(c\Lambda(\beta,2r)\big), \quad
\Lambda(\beta,r):=\exp\big(\beta\lambda(r)\big)
\end{equation}
for all $0<r\leqslant R/2$ and with some $\delta_{0}\in(0,1)$, and
\begin{equation}\label{eqSkr3.122505}
\int\limits_{0} \exp\big(-c\Lambda(\beta,r) \big)\frac{dr}{r}=+\infty,
\quad \lim\limits_{r\rightarrow0} r^{1-\delta_{0}}\exp\big(c\Lambda(\beta,r) \big)=0,
\end{equation}
then $u$ is continuous at $(x_{0},t_{0})$.
\end{theorem}

\begin{theorem}\label{th2.2}
Let $u\in \mathcal{B}_{1,g,\lambda}(Q_{R,R}(x_{0},t_{0}))$ and
let hypotheses $({\rm g})$, $({\rm g}_{1})$, $({\rm g}_{2}\psi_{2})$ be fulfilled.
Then there exist positive numbers $c$, $\bar{c}$, $\beta$ depending only on the data such that if
\begin{equation}\label{eqSkr3.132505}
\exp\big(8\Lambda_{1}(c,\beta,r)\big)\leqslant \left(\frac{3}{2}\right)^{1-\delta_{0}}
\exp\big(8\Lambda_{1}(c,\beta,2r)\big), \quad
\Lambda_{1}(c,\beta,r):=\exp\big( c\Lambda(\beta,r) \big)
\end{equation}
for all $0<r\leqslant R/2$ and with some $\delta_{0}\in (0,1)$, and
\begin{equation}\label{eqSkr3.142505}
\lim\limits_{r\rightarrow0}r^{1-\delta_{0}}\exp\big( 8\Lambda_{1}(c,\beta,r) \big)=0
\end{equation}
and
\begin{equation}\label{eqSkr3.152505}
\sum\limits_{i=0}^{\infty} \exp\big(-8\Lambda_{1}(c,\beta,r_{i}) \big)=+\infty,
\ \ r_{i}:=\bar{c}\,r_{i-1}
\exp\big( -8 \Lambda_{1}(c,\beta,r_{i-1}) \big),
\ \ r_{0}=\rho<R,
\end{equation}
then $u$ is continuous at $(x_{0},t_{0})$.
\end{theorem}

\begin{remark}
{\rm We note that the function $\lambda(r)=L\ln\ln\ln \dfrac{1}{r}$
satisfies conditions \eqref{eqSkr3.112505} and \eqref{eqSkr3.122505}
if $0<L<1/\beta$. We also note that the function
$\lambda(r)=L\ln\ln\frac{1}{8}\ln\ln\ln\ln \dfrac{1}{r}$
satisfies conditions \eqref{eqSkr3.132505}--\eqref{eqSkr3.152505}
if $0<L<1/\beta$. Indeed, to prove condition \eqref{eqSkr3.152505} we note
that by \eqref{eqSkr3.142505}
$$
r_{1}=\bar{c}\,\rho \exp\big(-8 \Lambda_{1}(c,\beta,\rho)\big)
\geqslant \rho^{\gamma_{0}}, \quad \gamma_{0}:=2-\delta_{0},
$$
if $\rho$ is sufficiently small, so $r_{j}\geqslant\rho^{\gamma_{0}^{j}}$,
$j=1,2,\ldots$, and hence
$$
j\geqslant \frac{1}{\ln\gamma_{0}} \ln \frac{\ln \frac{1}{r_{j}} }{\ln \frac{1}{\rho}}.
$$
This implies
\begin{multline*}
\sum\limits_{i=0}^{j} \exp\big(-8 \Lambda_{1}(c,\beta,r_{i}) \big)
\geqslant j \exp\big(-8 \Lambda_{1}(c,\beta,r_{j}) \big)
\\
\geqslant \frac{1}{\ln\gamma_{0}} \ln \frac{\ln \frac{1}{r_{j}} }{\ln \frac{1}{\rho}}
\left( \ln\ln\ln \frac{1}{r_{j}} \right)^{-1}\rightarrow +\infty
\quad \text{if} \ j\rightarrow +\infty.
\end{multline*}
}
\end{remark}

We note that in the parabolic case, the statements of Theorems
\ref{th2.1} and \ref{th2.2} differ from the elliptic case (Theorem \ref{elliptth2.1}).
The following examples show that additional conditions on the function $\psi$ arrives naturally.

\begin{example}\label{exmpl3.1}
{\rm The function
$$
g_{1}(x,t, {\rm v})={\rm v}^{p(x,t)-1}+{\rm v}^{q(x,t)-1}, \quad
(x,t, {\rm v})\in \mathbb{R}^{n}\times \mathbb{R}_{+} \times \mathbb{R}_{+},
$$
satisfies conditions
$({\rm g})$, $({\rm g}_{1})$  if
$1<p\leqslant p(x,t)\leqslant q(x,t)\leqslant q$ and
$$
\text{osc}\{p(x,t); Q_{r,r^{p^{-}}}(x_{0},t_{0})\}+
\text{osc}\{q(x,t); Q_{r,r^{p^{-}}}(x_{0},t_{0})\}
\leqslant
\frac{\lambda(r)}{\ln r^{-1}}, \quad p^{-}=\min\{2,p\}.
$$
Moreover, the function $g_{1}(x,t, {\rm v})$ satisfies condition  (${\rm g}_{2}\psi_{1}$)
with $\mu_{1}=p-2$, $\mu_{2}=q-2$, $\delta=0$ and $b_{0}=0$, if $p>2$.
In addition, if $1<p\leqslant p(x,t)\leqslant p_{1}<2<q_{1}\leqslant q(x,t)$
then the function $g_{1}(x,t, {\rm v})$ satisfies (${\rm g}_{2}\psi_{1}$) with
$\mu_{1}=(q_{1}-2)/2$, $\mu_{2}=q-2$, $\delta=0$ and
$b_{0}=\max \left\{1, \big(\frac{2-p+\mu_{1}}{q_{1}-2+\mu_{1}}\big)
^{\frac{1}{q_{1}-p_{1}}} \right\}$.
Finally, if $1<p\leqslant p(x,t)\leqslant q(x,t)\leqslant q<2$,
then the function $g_{1}(x,t, {\rm v})$ satisfies condition  (${\rm g}_{2}\psi_{2}$)
with $\mu_{3}=2-q$, $\mu_{4}=2-p$, $\delta=0$ and $b_{0}=0$.}
\end{example}
\begin{example}\label{exmpl3.3}
{\rm The function
$$
g_{2}(x,t, {\rm v})={\rm v}^{p(x,t)-1}\big( 1+\ln(1+{\rm v}) \big),
\quad (x,t, {\rm v})\in \mathbb{R}^{n}\times \mathbb{R}_{+}\times \mathbb{R}_{+},
$$
satisfies conditions $({\rm g})$, $({\rm g}_{1})$  if
$1<p\leqslant p(x,t)\leqslant p_{1}$ and
$$
\text{osc}\{p(x,t); Q_{r,r^{p^{-}}}(x_{0},t_{0})\}\leqslant \frac{\lambda(r)}{\ln r^{-1}}.
$$
If $p>2$ it satisfies condition (${\rm g}_{2}\psi_{1}$) with $\mu_{1}=p-2$, $\mu_{2}=p-1$, $\delta=0$
and $b_{0}=0$. If $p_{1}<2$ then the function $g_{3}(x,t, {\rm v})$
satisfies condition (${\rm g}_{2}\psi_{2}$) with $\mu_{3}=\dfrac{2-p}{2}$, $\mu_{4}=2-p$,
$\delta=0$ and $b_{0}=e^{\frac{2}{2-p}}-1$.
}
\end{example}
\begin{example}\label{exmpl3.2}
{\rm The function
$$
g_{3}(x,t, {\rm v})={\rm v}^{p-1}+a(x,t){\rm v}^{q-1}, \quad
(x,t, {\rm v})\in \mathbb{R}^{n}\times \mathbb{R}_{+}\times \mathbb{R}_{+},
$$
satisfies conditions $({\rm g})$, $({\rm g}_{1})$ if
$a(x,t)\geqslant0$, $1<p<q\leqslant p+\alpha$, $0<\alpha\leqslant1$,
$$
\text{osc}\{a(x,t); Q_{r,r^{p^{-}}}(x_{0},t_{0})\}
\leqslant
a_{0}r^{\alpha}e^{\lambda(r)} \ \ \text{and}
\ \ \lim\limits_{r\rightarrow0} r^{\alpha}e^{\lambda(r)}=0.
$$
In addition, if $p>2$ then it satisfies condition
(${\rm g}_{2}\psi_{1}$) with $\mu_{1}=p-2$, $\mu_{2}=q-2$, $\delta=0$
and $b_{0}=0$. In the case $q<2$ the function $g_{3}(x,t, {\rm v})$
satisfies hypothesis (${\rm g}_{2}\psi_{2}$) with $\mu_{3}=2-q$, $\mu_{4}=2-p$,
$\delta=0$ and $b_{0}=0$. Moreover, if $p<2<q$ and $a(x_{0},t_{0})=0$
then $g_{3}(x,t, {\rm v})$ satisfies \eqref{eqeqSkr2.8mu4} with
$\mu_{4}=2-p$. To check \eqref{eqeqSkr2.8mu3} we need
to obtain the inequality
$$
(p-2)u^{p-2}+(q-2)a(x,t)u^{q-2}\leqslant
-\frac{2-p}{2}(u^{p-2}+a(x,t)u^{q-2})
\ \ \text{for} \ (x,t)\in Q_{R,R}(x_{0},t_{0}),
$$
or the same
$$
a(x,t)\left(q-1-\frac{p}{2}\right)u^{q-p}\leqslant \frac{(2-p)}{2},
$$
this inequality will be fulfilled if $R$ is so small that
$$
a_{0}\left(q-1-\frac{p}{2}\right)R^{\alpha}e^{\lambda(R)}M^{q-p}\leqslant  \frac{2-p}{2},
$$
which yields \eqref{eqeqSkr2.8mu3} with $\mu_{3}=\dfrac{2-p}{2}$, $\delta=0$
and $b_{0}=0$.

Finally, if $2<p<q$ and $a(x_{0},t_{0})>0$, then 
the function $g_{3}(x,t, {\rm v})$ satisfies \eqref{eqSkr2.7}
with $\mu_{2}=q-2$. To check \eqref{psib0cond} we need to obtain the
inequality
$$
(p-2)u^{p-2}+(q-2)a(x_{0},t_{0})u^{q-2}\geqslant
\frac{q-2}{2}(u^{p-2}+a(x_{0},t_{0})u^{q-2}),
$$
or the same
$$
\frac{q-2}{2}\,a(x_{0},t_{0})u^{q-p}
\geqslant \frac{q}{2}+1-p.
$$
Choose $R$ from the condition
$a_{0}R^{\alpha} e^{\lambda(R)}=\frac{1}{2}a(x_{0},t_{0})$,
which yields \eqref{psib0cond} with
$$
\mu_{1}=\frac{q-2}{2}, \ \ \delta=\frac{\alpha}{q-p} \ \
\text{and} \ \
b_{0}=\bigg(\frac{2}{a_{0}}\, \frac{q+2-2p}{q-2}\bigg)^{\frac{1}{q-p}}.
$$
Note that this choice of $R$ guarantees that
$$
\frac{1}{2}a(x_{0},t_{0})\leqslant a(x,t)\leqslant \frac{3}{2}a(x_{0},t_{0})
\ \ \text{for all} \ (x,t)\in Q_{R,R}(x_{0},t_{0}).
$$
}
\end{example}
\begin{example}\label{exmpl3.4}
{\rm The function
$$
g_{4}(x,t, {\rm v})={\rm v}^{p-1}\big( 1+b(x,t)\ln(1+{\rm v}) \big),
\quad (x,t, {\rm v})\in \mathbb{R}^{n}\times\mathbb{R}_{+}\times\mathbb{R}_{+},
$$
satisfies conditions $({\rm g})$, $({\rm g}_{1})$ if
$p>1$, $b(x,t)\geqslant0$,
$$
\text{osc}\{b(x,t); Q_{r,r^{p^{-}}}(x_{0},t_{0})\}\leqslant \frac{Be^{\lambda(r)}}{\ln r^{-1}}
\ \ \text{and} \ \ \lim\limits_{r\rightarrow0}\frac{e^{\lambda(r)}}{\ln r^{-1}}=0.
$$

In addition, if $p>2$ then it satisfies condition (${\rm g}_{2}\psi_{1}$) with
$\mu_{1}=p-2$, $\mu_{2}=p-1$, $\delta=0$ and $b_{0}=0$.
Moreover, if $p<2$ then the function $g_{4}(x,t, {\rm v})$ satisfies
condition (${\rm g}_{2}\psi_{2}$) with
$$
\mu_{3}=\frac{2-p}{2}, \ \ \mu_{4}=2-p, \ \ \delta=0 \ \ \text{and}
\ \ b_{0}=e^{\frac{2}{2-p}}-1.
$$
}
\end{example}

\section{Parabolic equations with generalized Orlicz growth}

We consider the equation
\begin{equation}\label{g-growtheq}
u_{t}-{\rm div}\bigg( g(x,t,|\nabla u|)\frac{\nabla u}{|\nabla u|} \bigg)=0,
\quad (x,t)\in \Omega_{T},
\end{equation}
where $\Omega_{T}:= \Omega\times (0,T)$,  $\Omega$ is a bounded domain in $\mathbb{R}^{n}$
and $0<T<+\infty$.

Additionally, we assume that
\begin{itemize}
\item[(${\rm g}_{3}$)]
there exists $\mu>0$ such that,
for a.a. $(x,t)\in \Omega_{T}$ and for ${\rm w}\geqslant{\rm v}>0$,
$$
\frac{g(x, t, {\rm w})}{g(x, t, {\rm v})}\geqslant
\left( \frac{{\rm w}}{{\rm v}} \right)^{\mu}.
$$
\end{itemize}

We set $\mathcal{G}(x,t, {\rm v}):= g(x,t, {\rm v}){\rm v}$ for
$(x,t, {\rm v})\in \Omega_{T}\times (0,+\infty)$
and write $W^{1,\mathcal{G}}(\Omega_{T})$ for a class of functions $u: \Omega_{T}\rightarrow\mathbb{R}$
satisfying
$
\iint\limits_{\Omega_{T}} \mathcal{G}(x,t,|\nabla u|)\, dxdt<+\infty.
$
By $W^{1,\mathcal{G}}_{0}(\Omega_{T})$
we denote the set of functions $u\in W^{1,\mathcal{G}}(\Omega_{T})$ with the property
that, for every $t\in (0,T)$, the function $u(\cdot, t)$ has a compact support in $\Omega$.

We say that a measurable function $u: \Omega_{T}\rightarrow\mathbb{R}$ is a bounded weak
solution of Eq. \eqref{g-growtheq} in $\Omega_{T}$ if
$u\in C_{{\rm loc}}\big( 0,T; L^{2}_{{\rm loc}}(\Omega) \big)
\cap W^{1,\mathcal{G}}_{{\rm loc}}(\Omega_{T}) \cap  L^{\infty}(\Omega_{T})$,
$u_{t}\in L^{2}_{{\rm loc}}(\Omega_{T})$,
and for every subinterval $(t_{1}, t_{2})\subset (0,T)$,
the following equality holds:
\begin{equation}\label{parabintid}
\int\limits_{\Omega}u\varphi\,dx\bigg|_{t_{1}}^{t_{2}}+\int\limits_{t_{1}}^{t_{2}}\int\limits_{\Omega}
\bigg\{-u\varphi_{t}+g(x,t, |\nabla u|)\frac{\nabla u}{|\nabla u|}\nabla \varphi \bigg\}\, dxdt=0
\end{equation}
for any
$\varphi\in L^{\infty}(\Omega_{T})\cap W_{0}^{1,\mathcal{G}}(\Omega_{T})$,
$\varphi_{t}\in L^{2}(\Omega_{T})$.

\begin{remark}
{\rm We note that in the case when $g$ independent of $t$, condition
$u_{t}\in L^{2}_{{\rm loc}}(\Omega_{T})$ can be dropped. In this case
integral identity \eqref{parabintid} can be rewritten in terms of the Steklov averages.
}
\end{remark}

Testing identity \eqref{parabintid} by $\varphi=(u-k)_{\pm}\zeta^{c_{1}}(x)\chi(t)$,
where $\zeta$, $\chi$ are the same as in \eqref{eqSkrnew2.2}--\eqref{eqSkr2.4},
using the Young inequality and condition (${\rm g}_{1}$) similarly to that of
Section \ref{sectellB1cl}, we arrive at \eqref{eqSkrnew2.2}--\eqref{eqSkr2.4}.

Testing \eqref{parabintid} by
$$
\varphi=\frac{w_{k}\,\zeta^{c_{1}}}{G(x_{0},t_{0},w_{k}/r)}\,
\frac{t-t_{0}+\theta}{\theta},
$$
we obtain for all $t\in(t_{0}-\theta,t_{0})$
\begin{multline*}
\frac{\partial}{\partial t}
\int\limits_{B_{r}(x_{0})\times\{t\}}
\Phi_{k}(x_{0},t_{0},u)\,\frac{t-t_{0}+\theta}{\theta}\,
\zeta^{c_{1}}dx
\\
+
\int\limits_{B_{r}(x_{0})\times\{t\}}
\frac{\mathcal{G}(x,t,|\nabla(u-k)_{\pm}|)}{G(x_{0},t_{0},w_{k}/r)}
\left( \frac{\mathcal{G}(x_{0},t_{0},w_{k}/r)}{G(x_{0},t_{0},w_{k}/r)}-1 \right)
\,\frac{t-t_{0}+\theta}{\theta}\,
\zeta^{c_{1}}dx
\\
\leqslant \frac{\gamma}{\theta}
\int\limits_{B_{r}(x_{0})\times\{t\}}
\Phi_{k}(x_{0},t_{0},u)\,\zeta^{c_{1}}dx
+
\gamma c_{1} \int\limits_{B_{r}(x_{0})\times\{t\}}
\frac{g(x,t,|\nabla(u-k)_{\pm}|)w_{k}}{G(x_{0},t_{0},w_{k}/r)}\,
\frac{t-t_{0}+\theta}{\theta}\,|\nabla\zeta|\zeta^{c_{1}-1}dx.
\\
\end{multline*}
Condition (${\rm g}_{3}$) implies
$$
G(x_{0},t_{0},{\rm w})\leqslant \frac{1}{1+\mu}\,
g(x_{0},t_{0},{\rm w})\,{\rm w}
\ \ \text{for} \ {\rm w}>0,
$$
hence from the previous, by the Young inequality, we have
\begin{multline*}
\frac{\partial}{\partial t}
\int\limits_{B_{r}(x_{0})\times\{t\}}
\Phi_{k}(x_{0},t_{0},u)\,\frac{t-t_{0}+\theta}{\theta}\,
\zeta^{c_{1}}dx+\frac{\mu}{2}\int\limits_{B_{r}(x_{0})\times\{t\}}
\frac{\mathcal{G}(x,t,|\nabla(u-k)_{\pm}|)}{G(x_{0},t_{0},w_{k}/r)}
\,\frac{t-t_{0}+\theta}{\theta}\,
\zeta^{c_{1}}dx
\\
\leqslant \frac{\gamma}{\theta}
\int\limits_{B_{r}(x_{0})\times\{t\}}
\Phi_{k}(x_{0},t_{0},u)\,\zeta^{c_{1}}dx+
\frac{\gamma}{\sigma^{\gamma}r}\int\limits_{B_{r}(x_{0})\times\{t\}}
\frac{g(x,t,\frac{\gamma w_{k}}{\sigma r\zeta})}{G(x_{0},t_{0},w_{k}/r)}\,
\zeta^{c_{1}-1}dx.
\end{multline*}
Further, using conditions (${\rm g}_{1}$), (${\rm g}_{3}$), we have
for any $\varepsilon_{1}\in(0,1)$
\begin{multline*}
r\int\limits_{B_{r}(x_{0})\times\{t\}}
\frac{|\nabla(u-k)_{\pm}|}{w_{k}}\,\zeta^{c_{1}}dx
\leqslant
r\varepsilon_{1}^{\mu} \int\limits_{B_{r}(x_{0})\times\{t\}}
\frac{\mathcal{G}(x,t,|\nabla(u-k)_{\pm}|)}{w_{k}g(x,t,w_{k}/r)}\,
\zeta^{c_{1}}dx+ \frac{1}{\varepsilon_{1}}|B_{r}(x_{0})|
\\
\leqslant K_{2}(1+\mu)^{-1}\,\varepsilon_{1}^{\mu}re^{c_{6}\lambda(r)}
\int\limits_{B_{r}(x_{0})\times\{t\}}
\frac{\mathcal{G}(x,t,|\nabla(u-k)_{\pm}|)}{G(x_{0},t_{0},w_{k}/r)}
\,
\zeta^{c_{1}}dx+ \frac{1}{\varepsilon_{1}}|B_{r}(x_{0})|.
\end{multline*}
Choosing $\varepsilon_{1}$ from the condition
$\varepsilon_{1}^{\mu}=e^{-c_{6}\lambda(r)}$
and collecting the last two inequalities, we
arrive at the required \eqref{eqSkr2.9}.

Taking into account the previous arguments, we obtain
\begin{theorem}
Let $u$ be a solution to Eq. \eqref{g-growtheq},
fix point $(x_{0},t_{0})\in \Omega_{T}$ and assume that conditions
${\rm (g)}$, ${\rm (g_{1})}$, ${\rm (g_{2}\psi_{1})}$ and ${\rm (g_{3})}$
be fulfilled in some cylinder $Q_{R,R}(x_{0},t_{0})\subset\Omega_{T}$.
Assume also that conditions \eqref{eqSkr3.112505}, \eqref{eqSkr3.122505}
be fulfilled, then $u$ is continuous at $(x_{0},t_{0})$.
\end{theorem}
\begin{theorem}
Let $u$ be a solution to Eq. \eqref{g-growtheq},
fix point $(x_{0},t_{0})\in \Omega_{T}$ and assume that conditions
${\rm (g)}$, ${\rm (g_{1})}$, ${\rm (g_{2}\psi_{2})}$ and ${\rm (g_{3})}$
be fulfilled in some cylinder $Q_{R,R}(x_{0},t_{0})\subset\Omega_{T}$.
Assume also that conditions \eqref{eqSkr3.132505}--\eqref{eqSkr3.152505}
be fulfilled, then $u$ is continuous at $(x_{0},t_{0})$.
\end{theorem}

\section{De Giorgi type lemmas}



Fix $(\bar{x}, \bar{t})\in Q_{R,R}(x_{0},t_{0})$ and construct the cylinder
$Q_{r,\theta}(\bar{x},\bar{t})\subset Q_{R,R}(x_{0},t_{0})$,
$r\leqslant R^{\bar{\delta}}$, $\bar{\delta}>\delta$,
where $\delta$ is the parameter that was defined in assumptions
$({\rm g}_{2}\psi_{1})$ and $({\rm g}_{2}\psi_{2})$.
We denote by
$\mu^{\pm}$ and $\omega$ non-negative numbers such that
$\mu^{+}\geqslant \text{ess}\!\!\!\!\!\sup\limits_{Q_{r,\theta}(\bar{x},\bar{t})}
\!\!\! u$,
$\mu^{-}\leqslant \text{ess}\!\!\!\!\inf\limits_{Q_{r,\theta}(\bar{x},\bar{t})}
\!\!\!\! u$,
$\omega= \mu^{+}-\mu^{-}$.
%

Our main result of this section reads as follows
\begin{lemma}\label{thSkr3.1}
Let $u\in \mathcal{B}_{1,g,\lambda}(Q_{R,R}(x_{0},t_{0}))$
and $\xi\in(0,1)$, then there exists $\nu_{1}\in (0,1)$
depending only upon the data,  $\xi$, $\omega$,
$r$ and $\theta$ such that if
\begin{equation*}
|\{(x,t)\in Q_{r,\theta}(\bar{x},\bar{t}): u(x,t)\leqslant \mu^{-}+ \xi\omega\}|
\leqslant \nu |Q_{r,\theta}(\bar{x},\bar{t})| \ \ \text{for} \ \ \nu\leqslant\nu_{1},
\end{equation*}
then either
\begin{equation}\label{eqSkr3.2}
\xi\omega\leqslant (1+b_{0})r^{1-\delta/\bar{\delta}},
\end{equation}
or
\begin{equation}\label{eq5.2Skr2505}
u(x,t)\geqslant \mu^{-}+\frac{\xi\omega}{2} \ \ \text{for a.a.} \
(x,t)\in Q_{r/2,\theta/2}(\bar{x},\bar{t}).
\end{equation}
Likewise, if
\begin{equation}\label{eq5.3Skr2505}
|\{(x,t)\in Q_{r,\theta}(\bar{x},\bar{t}): u(x,t)\leqslant \mu^{-}+ \xi\omega\}|
\leqslant \nu |Q_{r,\theta}(\bar{x},\bar{t})| \ \ \text{for} \ \ \nu\leqslant\nu_{1},
\end{equation}
then \eqref{eqSkr3.2} holds true, or
\begin{equation}\label{eq5.4Skr2505}
u(x,t)\leqslant \mu^{+}-\frac{\xi\omega}{2} \ \ \text{for a.a.} \
(x,t)\in Q_{r/2,\theta/2}(\bar{x},\bar{t}).
\end{equation}
\end{lemma}
\begin{proof}
Further, we will suppose that
\begin{equation}\label{eqSkr3.4}
\xi\omega\geqslant (1+b_{0})r^{1-\delta/\bar{\delta}}.
\end{equation}
Let $(x_{1},t_{1})\in Q_{r/2,\theta/2}(\bar{x},\bar{t})$ be arbitrary.
For $j=0,1,2, \ldots$, we define the sequences
$$
r_{j}:=\frac{r}{4}(1+2^{-j}), \quad \theta_{j}:=\frac{\theta}{4}(1+2^{-j}), \quad
\bar{r}_{j}:=\frac{r_{j}+r_{j+1}}{2},
$$
$$
\bar{\theta}_{j}:=\frac{\theta_{j}+\theta_{j+1}}{2}, \quad
k_{j}:=\mu^{-}+\frac{\xi\omega}{2}+\frac{\xi\omega}{2^{j+1}}, \quad
B_{j}:=B_{r_{j}}(x_{1}),
$$
$$
\bar{B}_{j}:=B_{\bar{r}_{j}}(x_{1}), \quad Q_{j}:=Q_{r_{j},\theta_{j}}(x_{1},t_{1}),
\quad \bar{Q}_{j}:=Q_{\bar{r}_{j},\bar{\theta}_{j}}(x_{1},t_{1}),
$$
$$
A_{j,k_{j}}:=Q_{j}\cap\{u<k_{j}\}, \quad
\bar{A}_{j,k_{j}}:=\bar{Q}_{j}\cap\{u<k_{j}\}.
$$
Let $\zeta_{j}\in C_{0}^{\infty}(\bar{B}_{j})$ be such that
$0\leqslant\zeta_{j}\leqslant 1$, $\zeta_{j}=1$ in $B_{j+1}$ and
$|\nabla \zeta_{j}|\leqslant \gamma\, 2^{j}/r$.
Let also $\chi_{j}(t)$
be such that $\chi_{j}(t)=1$ for $t\geqslant t_{1}-\theta_{j+1}$,
$\chi_{j}(t)=0$ for $t\leqslant t_{1}-\bar{\theta}_{j}$,
$0\leqslant \chi_{j}(t) \leqslant1$ and
$\left|\frac{d\chi_{j}(t)}{dt}\right|\leqslant  \frac{\gamma\,2^{j}}{\theta}$.

Using \eqref{eqSkrnew2.5}, \eqref{eqSkr3.4}, and hypotheses (${\rm g}_{2}\psi_{1}$)
or (${\rm g}_{2}\psi_{2}$),
we can rewrite inequality \eqref{eqSkrnew2.3} with $\varepsilon=1$ in the form
\begin{multline}\label{eqSqr3.5}
\iint\limits_{\bar{A}_{j,k_{j}}} |\nabla (k_{j}-\max\{u, k_{j+1}\})_{+}|\,
\zeta_{j}^{c_{1}}\chi_{j}\,dxdt
\leqslant \gamma\, \frac{\xi\omega}{r}
e^{c_{2}\lambda(r)}|A_{j,k_{j}}|
\\
+\frac{\gamma\, 2^{j\gamma}}{g(x_{1},t_{1}, \xi\omega/r)}
\left\{ \frac{(\xi\omega)^{2}}{\theta}\, |A_{j,k_{j}}|+\frac{\xi\omega}{r}
\iint\limits_{A_{j,k_{j}}} g\left(x,t, \frac{K_{1}2^{j\gamma}\xi\omega}{r\zeta_{j}}\right)
\zeta_{j}^{c_{1}-1}\chi_{j}\, dxdt \right\}
\\
\leqslant
\gamma\, \frac{\xi\omega}{r}
e^{c_{2}\lambda(r)}|A_{j,k_{j}}|+
\gamma\, 2^{j\gamma} \frac{\xi\omega}{r}
\left\{  \frac{\xi\omega r}{\theta g(x_{1},t_{1}, \xi\omega/r)}+
e^{c_{6}\lambda(r)} \right\}
|A_{j,k_{j}}|
\\
\leqslant
\gamma\, 2^{j\gamma} \frac{\xi\omega}{r}  e^{c_{0}\lambda(r)}
\left(1+\frac{\xi\omega r}{\theta g(\bar{x}, \bar{t}, \xi\omega/r)} \right)|A_{j,k_{j}}|,
\hskip 4,65cm
\end{multline}
where $c_{0}:=\max\{c_{2},c_{6}\}$. Here we also assume that
$c_{1}\geqslant \max\{\mu_{2}, 2-\mu_{3}\}$.

Similarly inequality \eqref{eqSkr2.4} implies that
\begin{multline}\label{eqSqr3.6}
\sup\limits_{t_{1}-\bar{\theta}_{j}<t<t_{1}} \int\limits_{\bar{B}_{j}\times\{t\}}
(k_{j}-u)_{+}^{2}\,\zeta_{j}^{c_{1}}\chi_{j}\,dx
\\
\leqslant \gamma\, 2^{j\gamma} \bigg\{ \frac{(\xi \omega)^{2}}{\theta} |A_{j,k_{j}}| +
\frac{\xi \omega}{r} \iint\limits_{A_{j,k_{j}}}
g\bigg(x,t, \frac{K_{1}2^{j\gamma}\xi\omega}{r \zeta_{j}}\bigg)
\zeta_{j}^{c_{1}-1}\chi_{j}\, dxdt \bigg\}
\\
\leqslant \gamma\, 2^{j\gamma}\, \frac{\xi \omega}{r} e^{c_{0}\lambda(r)}
\bigg(1+ \frac{r\xi \omega}{\theta g(\bar{x}, \bar{t}, \xi \omega/r)}  \bigg)\,
g(\bar{x}, \bar{t}, \xi \omega/r)\, |A_{j,k_{j}}|.
\end{multline}

From \eqref{eqSqr3.5}, \eqref{eqSqr3.6}, by the Sobolev embedding theorem and
H\"{o}lder's inequality, we obtain
\begin{multline*}
(k_{j}-k_{j+1})|A_{j+1,k_{j+1}}|\leqslant \iint\limits_{\bar{A}_{j,k_{j}}}
(k_{j}-\max\{u, k_{j+1}\})_{+}\,\zeta_{j}^{c_{1}}\chi_{j}\,dxdt
\\
\leqslant \gamma |A_{j,k_{j}}|^{\frac{2}{n+2}} \bigg(
\sup\limits_{t_{1}-\bar{\theta}_{j}<t<t_{1}}
\int\limits_{\bar{B}_{j}\times \{t\}} (k_{j}-u)_{+}^{2}\,\zeta_{j}^{c_{1}}\chi_{j}\,dx
\bigg)^{\frac{1}{n+2}}
\\
\times \bigg(
\iint\limits_{\bar{A}_{j,k_{j}}}
\left|\nabla \big((k_{j}-\max\{u, k_{j+1}\})_{+}\,\zeta_{j}^{c_{1}}\big)\right|
\chi_{j}\,dxdt\bigg)^{\frac{n}{n+2}} \ \ \
\\
\leqslant\gamma\,2^{j\gamma}
\left(\frac{\xi\omega}{r} e^{c_{0}\lambda(r)} \right)^{\frac{n+1}{n+2}}
\left(1+\frac{r\xi\omega}{\theta g(\bar{x}, \bar{t}, \xi\omega/r)} \right)^{\frac{n+1}{n+2}}
\big[g(\bar{x}, \bar{t}, \xi\omega/r)\big]^{\frac{1}{n+2}}
|A_{j,k_{j}}|^{1+\frac{1}{n+2}},
\end{multline*}
which implies
\begin{multline*}
y_{j+1}:=\frac{|A_{j+1,k_{j+1}}|}{|Q_{r,\theta}(\bar{x}, \bar{t})|}
\\
\leqslant
\gamma\,2^{j\gamma} e^{c_{0}\frac{n+1}{n+2}\lambda(r)}
\left(1+\frac{r\xi\omega}{\theta g(\bar{x}, \bar{t}, \xi\omega/r)} \right)^{\frac{n+1}{n+2}}
 \left( \frac{\theta g(\bar{x}, \bar{t}, \xi\omega/r)}{r \xi\omega} \right)^{\frac{1}{n+2}}
y_{j}^{1+\frac{1}{n+2}}, \  j=0,1,2, \ldots.
\end{multline*}
Iterating this inequality, we obtain that $\lim\limits_{j\rightarrow\infty}y_{j}=0$,
provided that $y_{0}\leqslant \nu_{1}$ and $\nu_{1}$ is chosen to satisfy
\begin{equation}\label{eqSkr3.7}
\nu_{1}:= \gamma^{-1}  e^{-c_{0}(n+1)\lambda(r)}
\frac{r\xi\omega}{\theta g(\bar{x}, \bar{t}, \xi\omega/r)}
\left(1+\frac{r\xi\omega}{\theta g(\bar{x}, \bar{t}, \xi\omega/r)} \right)^{-n-1},
\end{equation}
which proves the lemma.
 \end{proof}

\begin{lemma}[De\,Giorgi type lemma involving initial data]\label{TheoremSkr3.2}
Assume that the hypotheses
$({\rm g})$, $({\rm g}_{1})$  be fulfilled,
and let $u\in \mathcal{B}_{1,g,\lambda}(Q_{R,R}(x_{0},t_{0}))$.
Fix $\xi\in(0,1)$, then
there exists $\nu_{2}\in (0,1)$ depending only on upon the data, $\xi$, $\omega$, $r$
and $\theta$ such that if
\begin{equation*}
u(x,t_{0}-\theta)\geqslant \mu^{-}+ \xi\omega \ \ \text{ for } \ x\in B_{r}(\bar{x}),
\end{equation*}
and
\begin{equation*}
|\{(x,t)\in Q_{r,\theta}(\bar{x}, \bar{t}): u(x,t)\leqslant \mu^{-}+\xi\omega\}|\leqslant
\nu |Q_{r,\theta}(\bar{x}, \bar{t})|,
\end{equation*}
for $\nu\leqslant\nu_{2}$, then either \eqref{eqSkr3.2} holds true, or
\begin{equation*}
u(x,t)\geqslant \mu^{-}+ \xi\omega/2 \ \ \text{ for a.a.} \
(x,t)\in Q_{r/2,\theta}(\bar{x}, \bar{t}).
\end{equation*}

Likewise, if
$$
u(x,t_{0}-\theta)\leqslant \mu^{+}- \xi\omega \ \ \text{ for } \ x\in B_{r}(\bar{x}),
$$
and
\begin{equation*}
|\{(x,t)\in Q_{r,\theta}(\bar{x}, \bar{t}): u(x,t)\geqslant \mu^{+}-\xi\omega\}|\leqslant
\nu |Q_{r,\theta}(\bar{x}, \bar{t})|,
\end{equation*}
for $\nu\leqslant\nu_{2}$, then either \eqref{eqSkr3.2} holds true, or
\begin{equation*}
u(x,t)\leqslant \mu^{+}- \xi\omega/2 \ \ \text{ for a.a.} \
(x,t)\in Q_{r/2,\theta}(\bar{x}, \bar{t}).
\end{equation*}
\end{lemma}
\begin{proof}
The proof is similar to that of Theorem \ref{thSkr3.1}. Taking $\chi(t)\equiv1$,
using inequalities \eqref{eqSkrnew2.3}, \eqref{eqSkr2.4} and repeat the same arguments
as in the previous proof, we prove the theorem with
\begin{equation}\label{eqSkr3.11}
\nu_{2}:=\gamma^{-1} e^{-c_{0}(n+1)\lambda(r)}
\frac{r\xi\omega}{\theta g(\bar{x}, \bar{t}, \xi\omega/r)}\,.
\end{equation}
\end{proof}


\section{Continuity in the ,,degenerate'' case, proof of Theorem \ref{th2.1}}


Fix number $\rho>0$ such that $\rho<R^{\bar{\delta}}$,
$\bar{\delta}>\max\{\delta/\delta_{0}, 1+\delta\}$ and construct the cylinder
$$
Q_{\rho}(x_{0},t_{0}):=B_{\rho}(x_{0})
\times (t_{0}-K_{2}c_{5}(1+b_{0})\,\rho^{2-\delta/\bar{\delta}}, t_{0})
\subset Q_{R,R}(x_{0},t_{0}),
$$
and set $\mu^{+}:=\text{ess}\!\!\!\!\!\sup\limits_{Q_{\rho}(x_{0},t_{0})}\!\!\!\! u$,
$\mu^{-}:= \text{ess}\!\!\!\!\!\!\inf\limits_{Q_{\rho}(x_{0},t_{0})}\!\!\!\!u$,
$\omega= \mu^{+}-\mu^{-}$.
If
$
\omega\geqslant 4(1+b_{0})\rho^{1-\delta/\bar{\delta}},
$
then, by \eqref{eqSkrnew2.6} and \eqref{psib0cond}, we have
\begin{multline*}
\psi\left(x_{0},t_{0}, \frac{\omega}{4\rho}\right)
\geqslant \psi(x_{0},t_{0}, (1+b_{0})\rho^{-\delta/\bar{\delta}} )
\\
=
\frac{\rho^{\delta/\bar{\delta}}g(x_{0},t_{0},(1+b_{0})\rho^{-\delta/\bar{\delta}} )}
{(1+b_{0})}\geqslant \frac{\rho^{\delta/\bar{\delta}}g(x_{0},t_{0},1)}{1+b_{0}}
\geqslant \frac{\rho^{\delta/\bar{\delta}}}{c_{5}(1+b_{0})},
\end{multline*}
and therefore
$$
Q_{\rho,\bar{\theta}}(x_{0},t_{0})\subset Q_{\rho}(x_{0},t_{0}), \quad
\bar{\theta}:=\frac{K_{2}\rho^{2}}{\psi\left(x_{0},t_{0}, \dfrac{\omega}{4\rho}\right)}.
$$
For fixed positive numbers $s_{\ast}:=s_{\ast}(\rho)<s^{\ast}:=s^{\ast}(\rho)$,
which will be specified
later, set
$$
\theta_{\ast}:= \frac{K_{2}\rho^{2}}{\psi\left(x_{0},t_{0},
\dfrac{\omega}{2^{s_{\ast}+2} \rho}\right)},
$$
and construct the cylinder
$Q_{\rho, \theta_{\ast}}(x_{0},t_{0})\subset Q_{\rho, \bar{\theta}}(x_{0},t_{0})$.
Consider also the cylinders
$Q_{\rho, \eta}(x_{0},\bar{t})\subset Q_{\rho,\theta_{\ast}}(x_{0},t_{0})$, where
$$
\eta:=\frac{K_{2}\rho^{2}}{4\psi\left(x_{0},t_{0},
\dfrac{\omega}{4 \rho}\right)} \ \ \text{and} \ \
t_{0}-\theta_{\ast}\leqslant \bar{t}-\eta< \bar{t}\leqslant t_{0}.
$$
The following two alternative cases are possible:

\textsl{First alternative}. There exists a cylinder
$Q_{\rho, \eta}(x_{0},\bar{t})\subset Q_{\rho,\theta_{\ast}}(x_{0},t_{0})$
such that
\begin{equation*}
|\{(x,t)\in Q_{\rho, \eta}(x_{0},\bar{t}): u(x,t)\leqslant \mu^{-}+\omega/2\}|\leqslant
\overline{\nu} e^{-c_{0}(2n+3)\lambda(\rho)}
|Q_{\rho, \eta}(x_{0},\bar{t})|,
\end{equation*}
where $\overline{\nu}$ is a sufficiently small positive number which will be chosen later
depending only upon the data.

\textsl{Second alternative}. For all cylinders
$Q_{\rho, \eta}(x_{0},\bar{t})\subset Q_{\rho,\theta_{\ast}}(x_{0},t_{0})$
the following inequality holds:
\begin{equation*}\label{eqSkr4.3}
|\{(x,t)\in Q_{\rho, \eta}(x_{0},\bar{t}): u(x,t)\geqslant \mu^{+}-\omega/2\}|\leqslant
\left(1-\overline{\nu}e^{-c_{0}(2n+3)\lambda(\rho)} \right)
|Q_{\rho, \eta}(x_{0},\bar{t})|,
\end{equation*}

Further we will assume that
\begin{equation}\label{eqSkr4.4}
\omega\geqslant 2^{s^{\ast}+2}(1+b_{0})\rho^{1-\delta/\bar{\delta}}.
\end{equation}

\subsection{Analysis of the first alternative}

By Theorem \ref{thSkr3.1} with $r=\rho$, $\theta=\eta$, $\xi=1/2$,
$(\bar{x}, \bar{t})=(x_{0}, \bar{t})$, assuming that $\rho$ is small enough,
and since the number $\nu_{1}$ defined in \eqref{eqSkr3.7} satisfies
$$
\nu_{1}=\gamma^{-1}e^{-c_{0}(n+1)\lambda(\rho)}
\frac{r\omega}{\eta\, g(x_{0}, \bar{t}, \omega/2r)}\,
\frac{1}{\left(1+ \dfrac{r\omega}{\eta\, g(x_{0}, \bar{t}, \omega/2r)} \right)^{n+1}}
\geqslant
\gamma^{-1}e^{-c_{0}(2n+3)\lambda(\rho)},
$$
we conclude that
\begin{equation}\label{eqSkr4.5}
u(x,\bar{t})\geqslant \mu^{-}+\omega/4 \ \ \text{for all} \ \ x\in B_{\rho/2}(x_{0}),
\end{equation}
provided $\overline{\nu}$ is chosen such that $\overline{\nu}=\gamma^{-1}$.
\begin{lemma}
For any $\nu\in(0,1)$ there exists $s^{\ast}$ depending only upon the data, $s_{\ast}$
and $\rho$ such that 
\begin{equation}\label{eqSkr4.6}
\left| \left\{x\in B_{\rho/4}(x_{0}): u(x,t)\leqslant \mu^{-}+\omega/2^{s^{\ast}} \right\} \right|
\leqslant \nu \left|B_{\rho/4}(x_{0})\right| \quad
\text{for all} \ \ t\in (\bar{t}, t_{0}).
\end{equation}
\end{lemma}

\begin{proof}
We use inequality \eqref{eqSkr2.4} with
$\sigma=1/2$, $k=\mu^{-}+2^{1-s^{\ast}}\omega$,
$r=\rho$, $\chi(t)\equiv1$.
By \eqref{eqSkrnew2.5}, \eqref{eqSkr4.4}, \eqref{eqSkr4.5}
and $({\rm g}_{2}\psi_{1})$, we obtain
\begin{multline*}
\int\limits_{B_{\rho/4}(x_{0})\times\{t\}}
(u-k)_{-}^{2}\,dx
\leqslant
\frac{\gamma\,\omega}{2^{s^{\ast}}\rho}
\iint\limits_{A^{-}_{k,\rho/2, \theta_{\ast}}}
g\left(x,t, \gamma(u-k)_{-}/\rho\right)dxdt
\\
\leqslant \gamma e^{c_{0}\lambda(\rho)}
\left( \frac{\omega}{2^{s^{\ast}}} \right)^{2}
\frac{\psi(x_{0},t_{0}, 2^{-s^{\ast}}\rho^{-1}\omega)}
{\psi(x_{0},t_{0}, 2^{-s_{\ast}}\rho^{-1}\omega)}\, |B_{\rho}(x_{0})|
\leqslant
\gamma e^{c_{0}\lambda(\rho)}
\left( \frac{\omega}{2^{s^{\ast}}} \right)^{2}
2^{-(s^{\ast}-s_{\ast})\mu_{1}}|B_{\rho}(x_{0})|,
\end{multline*}
which implies that, for all $t\in(\bar{t},t_{0})$,
$$
\left| \left\{ x\in B_{\rho/4}(x_{0}):
u(x,t)\leqslant \mu^{-}+ \omega/2^{s^{\ast}} \right\} \right|
\leqslant
\gamma
2^{-(s^{\ast}-s_{\ast})\mu_{1}} e^{c_{0}\lambda(\rho)}|B_{\rho}(x_{0})|.
$$
From this we arrive at \eqref{eqSkr4.6} with
$$
2^{s^{\ast}}=2^{s_{\ast}}\nu^{-1/\mu_{1}} e^{c_{0}\lambda(\rho)/\mu_{1}},
$$
where the number $s_{\ast}$ will be specified later.
This completes the proof of the lemma.
\end{proof}

By Theorem \ref{TheoremSkr3.2} with $(\bar{x}, \bar{t})=(x_{0},t_{0})$, $\xi=2^{-s^{\ast}}$,
$\eta\leqslant \theta \leqslant \theta_{\ast}$, since the number $\nu_{2}$ defined in
\eqref{eqSkr3.11} satisfies
$$
\nu_{2}\geqslant \gamma^{-1} e^{-c_{0}(n+1)\lambda(\rho)}
\frac{\psi(x_{0},t_{0},\omega/4\rho)}{\psi(x_{0},t_{0},\omega/2^{s^{\ast}})}
\geqslant
\gamma^{-1} e^{-c_{0}(n+1)\lambda(\rho)},
$$
and choosing $\nu$ from the condition
$$
\nu=\gamma^{-1} e^{-c_{0}(n+1)\lambda(\rho)},
$$
we arrive at
$$
u(x,t)\geqslant \mu^{-}+\omega/2^{s^{\ast}+1}
\quad \text{for a.a.} \ (x,t)\in Q_{\rho/8, \eta/8}(x_{0},t_{0}).
$$
From this we obtain that
\begin{equation}\label{eqSkr4.7}
{\rm osc}\big\{u; Q_{\rho/8,\,\eta/8}(x_{0},t_{0})\big\}\leqslant
(1-2^{-s^{\ast}-1})\omega,
\end{equation}
with the number $s^{\ast}$ defined by
\begin{equation}\label{eqSkr4.8}
2^{s^{\ast}}=\gamma\, 2^{s_{\ast}}
e^{\beta_{0}\lambda(\rho)}, \quad
\beta_{0}=\frac{c_{0}(n+2)}{\mu_{1}},
\end{equation}
where $s_{\ast}$ will be specified later.


\subsection{Analysis of the second alternative}

\begin{lemma}
Fix a cylinder $Q_{\rho, \eta}(x_{0},\bar{t})$, then there exists
$$
\tilde{t}\in \left(\bar{t}-\eta,\, \bar{t}-\eta\,\frac{\bar{\nu}}{2}
e^{-c_{0}(2n+3)\lambda(\rho)}  \right)
$$
such that
\begin{equation}\label{eqSkr4.9}
\left| \left\{ x\in B_{\rho}(x_{0}): u(x, \tilde{t})\geqslant \mu^{+}-\omega/2 \right\} \right|
\leqslant
\frac{1-\bar{\nu}e^{-c_{0}(2n+3)\lambda(\rho)} }
{1-\dfrac{\bar{\nu}}{2} e^{-c_{0}(2n+3)\lambda(\rho)}}\,
|B_{\rho}(x_{0})|.
\end{equation}
\end{lemma}

\begin{proof}
Suppose that the statement of the lemma is false, than for all
$$
t\in \left(\bar{t}-\eta,\, \bar{t}-\eta\,\frac{\bar{\nu}}{2}
e^{-c_{0}(2n+3)\lambda(\rho)}  \right)
$$
there holds
\begin{equation*}
\left| \left\{ x\in B_{\rho}(x_{0}): u(x, t)\geqslant \mu^{+}-\omega/2 \right\} \right|
>
\frac{1-\bar{\nu}e^{-c_{0}(2n+3)\lambda(\rho)} }
{1-\dfrac{\bar{\nu}}{2} e^{-c_{0}(2n+3)\lambda(\rho)}}\,
|B_{\rho}(x_{0})|,
\end{equation*}
and hence
\begin{multline*}
\left| \left\{ (x,t)\in Q_{\rho, \eta}(x_{0}, \bar{t}):
u(x,t)\geqslant \mu^{+}-\omega/2 \right\} \right|
\\
\geqslant
\int\limits_{\bar{t}-\eta}^{\bar{t}
-\eta\frac{\bar{\nu}}{2} e^{-c_{0}(2n+3)\lambda(\rho)}}
\left| \left\{ x\in B_{\rho}(x_{0}): u(x, t)\geqslant \mu^{+}-\omega/2 \right\} \right|dt
\\
>\left(1-\bar{\nu}e^{-c_{0}(2n+3)\lambda(\rho)} \right)
\left|  Q_{\rho, \eta}(x_{0}, \bar{t}) \right|,
\end{multline*}
reaching a contradiction. The lemma is proved.
\end{proof}

We will assume later that $\bar{\nu}$ is so small that inequality \eqref{eqSkr4.9}
yields
\begin{equation*}
\left| \left\{ x\in B_{\rho}(x_{0}): u(x, \tilde{t})\geqslant \mu^{+}-\omega/2 \right\} \right|\leqslant
\left( 1- \frac{\bar{\nu}^{2}}{2} e^{-c_{0}(2n+3)\lambda(\rho)}
 \right) |B_{\rho}(x_{0})|.
\end{equation*}

\begin{lemma}
There exists $s_{1}>1$ depending only upon the data and $\rho$ such that
\begin{equation}\label{eqSkr4.10}
\left| \left\{ x\in B_{\rho}(x_{0}): u(x, t)\geqslant \mu^{+}-\omega/2^{s_{1}} \right\} \right|
\leqslant \left( 1- \frac{\bar{\nu}^{4}}{8}
e^{-c_{0}(2n+3)\lambda(\rho)}
 \right) |B_{\rho}(x_{0})|
\end{equation}
for all $t\in (\tilde{t}, \tilde{t}+\eta)$.
\end{lemma}
\begin{proof}
We use inequality \eqref{eqSkr2.4} with $k=\mu^{+}-\omega/2^{s}$, $r=\rho$,
$1<s<s_{1}$, $\chi(t)\equiv1$.
Using inequality \eqref{eqSkr4.4}, we estimate the last term on the right-hand side of
\eqref{eqSkr2.4} as follows
\begin{multline*}
\rho^{-1}\iint\limits_{A^{+}_{k,\rho,\eta}}
g\left( x,t, \frac{K_{1}(u-k)_{+}}{\sigma\rho\zeta} \right) (u-k)_{+}\,\zeta^{c_{1}-1}dxdt
\\
\leqslant \frac{\gamma}{\sigma^{\mu_{2}+1}}\,
e^{c_{6}\lambda(\rho)}\frac{\omega\,\eta}{2^{s}\rho}\,
g\left(x_{0},t_{0}, \frac{\omega}{2^{s}\rho} \right) |B_{\rho}(x_{0})|
\leqslant
\frac{\gamma\, \sigma^{-\mu_{2}-1}}{2^{s(\mu_{1}-1)}}\,
e^{c_{0}\lambda(\rho)}
\omega^{2}.
\end{multline*}
We estimate the term on the left-hand side of \eqref{eqSkr2.4} by
$$
\int\limits_{B_{\rho(1-\sigma)}(x_{0})} (u-k)_{+}^{2}dx
\geqslant
(\omega/2^{s})^{2} (1-2^{s-s_{1}})^{2}
\Big( \left|  \left\{ x\in B_{\rho}(x_{0}):
u(x,t)\geqslant \mu^{+}-\omega/2^{s_{1}} \right\}\right|-n\sigma |B_{\rho}(x_{0})| \Big).
$$
From this and \eqref{eqSkr2.4} we arrive at
\begin{multline*}
\left|  \left\{ x\in B_{\rho}(x_{0}):
u(x,t)\geqslant \mu^{+}-\omega/2^{s_{1}} \right\}\right|
\\
\leqslant  \Bigg\{n\sigma+ (1-2^{s-s_{1}})^{-2} \left( 1-\frac{\bar{\nu}^{2}}{2}
e^{-c_{0}(2n+3)\lambda(\rho)} \right)
+\frac{\gamma\,\sigma^{-1-\mu_{2}}}{2^{s(\mu_{1}-1)}} (1-2^{s-s_{1}})^{-2}
e^{c_{0}\lambda(\rho)} \Bigg\}
|B_{\rho}(x_{0})|.
\end{multline*}

First choose $s_{1}$ from the condition
$$
(1-2^{s-s_{1}})^{-2}=1+\frac{\bar{\nu}^{2}}{2}
e^{-c_{0}(2n+3)\lambda(\rho)},
$$
then choose $\sigma$ from the condition
$$
n\sigma=\frac{\bar{\nu}^{2}}{16}
e^{-2c_{0}(2n+3)\lambda(\rho)},
$$
and finally, choose $s$ such that
$$
\frac{\gamma\,\sigma^{-1-\mu_{2}}}{2^{s(\mu_{1}-1)}}
e^{c_{0}\lambda(\rho)}=
\frac{\bar{\nu}^{2}}{16}e^{-2c_{0}(2n+3)\lambda(\rho)},
$$
we arrive at \eqref{eqSkr4.10} with $s_{1}$ defined by the equality
\begin{equation}\label{eqSkr4.11}
2^{-s_{1}}=2^{-s}\left[ 1- \left( 1+ \frac{\bar{\nu}^{2}}{2}
e^{c_{0}(2n+3)\lambda(\rho)}  \right)^{-1/2} \right]
\geqslant\gamma^{-1}(\bar{\nu})
e^{-\beta_{1}\lambda(\rho)},
\end{equation}
$$
\beta_{1}=c_{0}\left( 2n+3+ \frac{1+(3+\mu_{2})(2n+3)}{\mu_{1}-1} \right).
$$
The lemma is proved.
\end{proof}

\begin{lemma}
For any $\nu_{2}\in (0,1)$ there exists $s_{\ast}$ depending only upon the data,
$\rho$ and $\nu_{2}$ such that
\begin{equation}\label{eqSkr4.12}
\left| \left\{ Q_{\rho,\theta_{\ast}} (x_{0},t_{0}):
u(x,t)\geqslant \mu^{+}-\omega/2^{s_{\ast}} \right\} \right|
\leqslant
\nu_{2} e^{-c_{0}(n+1)\lambda(\rho)}
|Q_{\rho,\theta_{\ast}} (x_{0},t_{0})|.
\end{equation}
\end{lemma}
\begin{proof}
For $s=s_{1}+1, \ldots , s_{\ast}$, we set $k_{s}=\mu^{+}-\omega/2^{s}$.
We use inequality \eqref{eqSkrnew2.2} with $\sigma=1/2$ and
$$
\varepsilon=\left( \frac{|A^{+}_{k_{s},\rho,\theta_{\ast}}\setminus
A^{+}_{k_{s+1},\rho,\theta_{\ast}}|}
{|Q_{\rho,\theta_{\ast}}(x_{0},t_{0})|} \right)^{1/(1+c_{4})}.
$$
We also choose $\chi(t)$ such that $0\leqslant \chi(t)\leqslant 1$,
$\chi(t)=0$ for $t\leqslant t_{0}-2\theta_{\ast}$, $\chi(t)=1$ for
$t\geqslant t_{0}-\theta_{\ast}$ and $|\chi_{t}|\leqslant 2/\theta_{\ast}$.
By these choices and by \eqref{eqSkrnew2.5}, \eqref{eqSkr2.7} and \eqref{eqSkr4.4},
inequality \eqref{eqSkrnew2.2} implies
$$
\iint\limits_{A^{+}_{k_{s},\rho,\theta_{\ast}}\setminus A^{+}_{k_{s+1},\rho,\theta_{\ast}}}
|\nabla u|\, dxdt \leqslant \frac{\gamma\,\omega}{2^{s}\rho}
\left( \frac{|A^{+}_{k_{s},\rho,\theta_{\ast}}\setminus
A^{+}_{k_{s+1},\rho,\theta_{\ast}}|}
{|Q_{\rho,\theta_{\ast}}(x_{0},t_{0})|} \right)^{c_{4}/(1+c_{4})}
e^{c_{0}\lambda(\rho)} |Q_{\rho,\theta_{\ast}}(x_{0},t_{0})|.
$$
By \eqref{eqSkr4.10} and Lemma \ref{lem2.1DGPn},
from the previous we obtain
\begin{multline*}
\frac{\omega}{2^{s+1}}\,|A^{+}_{k_{s},\rho, \theta_{\ast}}|\leqslant
\gamma(\bar{\nu})\,\rho\, e^{2c_{0}(2n+3)\lambda(\rho)}
\iint\limits_{A^{+}_{k_{s},\rho,\theta_{\ast}}\setminus A^{+}_{k_{s+1},\rho,\theta_{\ast}}}
|\nabla u|\, dxdt
\\
\leqslant \gamma(\bar{\nu})
\left( \frac{|A^{+}_{k_{s},\rho,\theta_{\ast}}\setminus
A^{+}_{k_{s+1},\rho,\theta_{\ast}}|}
{|Q_{\rho,\theta_{\ast}}(x_{0},t_{0})|} \right)^{c_{4}/(1+c_{4})}
e^{c_{0}(4n+7)\lambda(\rho)} \frac{\omega}{2^{s}\rho}\,
|Q_{\rho,\theta_{\ast}}(x_{0},t_{0})|.
\end{multline*}
Summing up this inequality for $s=s_{1}+1, \ldots, s_{\ast}$ and choosing $s_{\ast}$
from the condition
$$
\frac{\gamma(\bar{\nu}) e^{c_{0}(4n+7)\lambda(\rho)}}
{(s_{\ast}-s_{1})^{c_{4}/(c_{4}+1)}}=
\nu_{2} e^{-c_{0}(n+1)\lambda(\rho)},
$$
we arrive at \eqref{eqSkr4.12} with
\begin{equation}\label{eqSkr4.13}
s_{\ast}=s_{1}+\gamma(\bar{\nu},\nu_{2})
e^{\beta_{2}\lambda(\rho)}, \quad
\beta_{2}=c_{0}(5n+8)(c_{4}+1)/c_{4},
\end{equation}
and $s_{1}$ was defined in \eqref{eqSkr4.11}.
The lemma is proved.
\end{proof}

By Theorem \ref{thSkr3.1} with $(\bar{x}, \bar{t})=(x_{0},t_{0})$, $\xi=2^{-s_{\ast}}$,
$\theta=\theta_{\ast}$, $\nu_{2}=1/\gamma$, inequality \eqref{eqSkr4.12} implies
$$
u(x,t)\leqslant \mu^{+}-\omega/2^{s_{\ast}+1} \ \ \text{ for a.a.} \ \
(x,t)\in Q_{\rho/2, \theta_{\ast}/2}(x_{0},t_{0}),
$$
which yields
\begin{equation}\label{eqSkr4.14}
{\rm osc}\big\{u; Q_{\rho/2,\,\theta_{\ast}/2}(x_{0},t_{0})\big\}\leqslant
(1-2^{-s^{\ast}-1})\omega.
\end{equation}

Collecting estimates \eqref{eqSkr4.7} and \eqref{eqSkr4.14}, we obtain that
\begin{equation}\label{eqSkr4.15}
{\rm osc}\big\{u; Q_{\rho/8,\,\eta/8}(x_{0},t_{0})\big\}\leqslant
(1-2^{-s^{\ast}-1})\,\omega,
\end{equation}
where $s^{\ast}$ defined by the equality
\begin{equation}
2^{s^{\ast}}=\gamma e^{\beta_{0}\lambda(\rho)}
2^{ \gamma(\bar{\nu}) \Lambda(\beta_{1},\rho) +
\gamma(\bar{\nu}, \nu_{2}) \Lambda(\beta_{2},\rho)},
\end{equation}
and numbers $\beta_{0}$, $\beta_{1}$, $\beta_{2}$ were defined in \eqref{eqSkr4.8},
\eqref{eqSkr4.11} and \eqref{eqSkr4.13}.

Choose $\rho$ small enough such that
$$
\begin{aligned}
2^{s^{\ast}}&=\gamma e^{\beta_{0}\lambda(\rho)}
2^{ \gamma(\bar{\nu}) \Lambda(\beta_{1},\rho) +
\gamma(\bar{\nu}, \nu_{2}) \Lambda(\beta_{2},\rho)}
\\
&\leqslant
4^{ \gamma(\bar{\nu}) \Lambda(\beta_{1},\rho) +
\gamma(\bar{\nu}, \nu_{2}) \Lambda(\beta_{2},\rho)}
\leqslant  \exp\big(\gamma\Lambda(\beta,\rho)\big),
\end{aligned}
$$
where $\beta$ depends only on the known data.
Then inequality \eqref{eqSkr4.15} can be rewritten in the form
\begin{equation}\label{eqSkr4.18}
{\rm osc}\big\{u; Q_{\rho/8,\,\eta/8}(x_{0},t_{0})\big\}\leqslant
\left( 1-\frac{1}{2} \exp\big(-\gamma\Lambda(\beta,\rho)\big) \right)\omega.
\end{equation}

Define the sequences
$$
r_{j}:=\frac{\rho}{8^{j+1}}, \quad
\omega_{j+1}:= \left(1-\frac{1}{2}
\exp\big(-\gamma\Lambda(\beta,r_{j})\big)\right)\omega_{j},
$$
$$
\theta_{j}:=\frac{r_{j}^{2}}{\psi(x_{0},t_{0}, \omega_{j}/4r_{j})},
\quad
Q_{j}:=Q_{r_{j}, \theta_{j}} (x_{0},t_{0}), \quad
j=0,1,2, \ldots .
$$
Since
$$
\frac{\omega_{j}}{r_{j}}=\frac{\omega_{j+1}}
{8r_{j+1} \left( 1- \frac{1}{2} \exp\big(-\gamma\Lambda(\beta,r_{j})\big) \right) }
\leqslant \frac{\omega_{j+1}}{r_{j+1}}, \quad j=0,1,2, \ldots ,
$$
from \eqref{eqSkr4.18} we have
$$
{\rm osc}\{u; Q_{1}\}\leqslant \omega_{1}.
$$
Repeating the previous procedure, by our choices we obtain
that either
$$
\omega_{j}\leqslant 4(1+b_{0})r_{j}^{1-\delta/\bar{\delta}}
\exp\big( \gamma\Lambda(\beta,r_{j}) \big) \leqslant
4(1+b_{0})r_{j}^{1-\delta_{0}}
\exp\big( \gamma\Lambda(\beta,r_{j}) \big)
$$
or
$$
{\rm osc}\left\{u; Q_{j}\right\}\leqslant \omega_{j}, \quad
j=1,2, \ldots .
$$
Note that by \eqref{eqSkrnew2.6}, \eqref{psib0cond} we have an inclusion
$\widetilde{Q}_{j}\subset Q_{j}$, where
$$
\widetilde{Q}_{j}:= B_{r_{j}}(x_{0})\times
\bigg( t_{0}-\frac{r_{j}^{2}}{\psi(x_{0},t_{0},\frac{M}{2r_{j}})},\, t_{0}  \bigg).
$$
Iterating the previous inequality and using \eqref{eqSkr3.112505}, we obtain
\begin{multline*}
{\rm osc}\{u; \widetilde{Q}_{j}\}\leqslant
\omega \prod\limits_{i=0}^{j-1}
\left( 1- \frac{1}{2} \exp\big(-\gamma\Lambda(\beta,r_{i})\big) \right)
+\gamma(1+b_{0})\rho^{1-\delta_{0}}\exp\big( \gamma\Lambda(\beta,\rho) \big)
\\
\leqslant
\omega\exp\left( -\frac{1}{2} \sum\limits_{i=0}^{j-1}
\exp\big(-\gamma\Lambda(\beta,r_{i})\big) \right)
+\gamma(1+b_{0})\rho^{1-\delta_{0}}\exp\big( \gamma\Lambda(\beta,\rho) \big),
\end{multline*}
which yields  
$$
{\rm osc}\{u; \widetilde{Q}_{j}\}\leqslant \omega
\exp \left( -\gamma\int\limits_{2r_{j}}^{\rho}
\exp\big(-\gamma \Lambda(\beta,s) \big)
\frac{ds}
{s } \right)+
\gamma(1+b_{0})\rho^{1-\delta_{0}} \exp\big( \gamma \Lambda(\beta,\rho) \big),
$$
which implies continuity of
$u\in \mathcal{B}_{1,g,\lambda}(Q_{R,R}(x_{0},t_{0}))$
at $(x_{0},t_{0})$ under conditions \eqref{eqSkr3.112505}, \eqref{eqSkr3.122505}.
This completes the proof of Theorem \ref{th2.1}. 


\section{Continuity in the ,,singular'' case, proof of Theorem \ref{th2.2} }

We fix a positive number $\bar{\delta}$ by the condition
$\bar{\delta}>\max\{\delta/\delta_{0}, 1+\delta_{0}\}$.
For $0<\rho<R^{\bar{\delta}}$, we construct the cylinder
$Q_{\rho,2c_{5}c_{10}M\rho}(x_{0},t_{0})\subset
 Q_{R,R}(x_{0},t_{0})$.
We set $Q_{\rho}(x_{0},t_{0}):=Q_{\rho,2c_{5}c_{10}M\rho}(x_{0},t_{0})$, \
$\mu^{+}:=\text{ess}\!\!\!\!\sup\limits_{Q_{\rho}(x_{0},t_{0})}\!\!\! u$, \
$\mu^{-}:=\text{ess}\!\!\!\!\inf\limits_{Q_{\rho}(x_{0},t_{0})}\!\!\! u$, \
$\omega:=\mu^{+}-\mu^{-}$.

If
$$
\eta\omega\geqslant (1+b_{0})\rho^{1-\delta/\bar{\delta}}
\geqslant (1+b_{0})\rho R^{-\delta}
$$
for some $\eta\in(0,1)$, then, by \eqref{eqeqSkr2.8mu3} and \eqref{eqSkrnew2.6}, we have
$$
\frac{\rho^{2}}{\psi(x_{0},t_{0}, \eta\omega/\rho)}\leqslant
\frac{c_{10}\rho^{2}}{\psi(x_{0},t_{0}, \omega/\rho)}
\leqslant
\frac{c_{10}\rho\,\omega}{g(x_{0},t_{0},1)}
\leqslant 2c_{5}c_{10}M\rho,
$$
and the following inclusion is true:
\begin{equation}\label{defthetac6c10}
Q_{\rho, \theta}\subset Q_{\rho}(x_{0},t_{0}),
\quad \text{where} \ \
\theta:=\frac{\rho^{2}e^{-c_{0}\lambda(\rho)}}{\psi(x_{0},t_{0}, \eta\omega/\rho)}.
\end{equation}

The following two alternative cases are possible:
\begin{equation}\label{firstalterpsi2}
\left|\left\{ x\in B_{\rho}(x_{0}): u(x,t_{0}-\theta) \leqslant\mu^{-}+\omega/2 \right\} \right|
\leqslant \frac{1}{2} |B_{\rho}(x_{0})|
\end{equation}
or
$$
\left|\left\{ x\in B_{\rho}(x_{0}):u(x,t_{0}-\theta)\geqslant \mu^{+}-\omega/2 \right\} \right|
\leqslant \frac{1}{2} |B_{\rho}(x_{0})|.
$$
Both cases can be considered completely similar.
Assume, for example, the first one.
\begin{lemma}\label{lemaoneta}
There exists $\eta\in (0,\, 1/2)$ depending only upon the data
such that if
\begin{equation}\label{etaomeggeqb0}
\eta\omega\geqslant (1+b_{0})\rho^{1-\delta/\bar{\delta}},
\end{equation}
then
\begin{equation}\label{eqSkr4.20}
\left|\left\{ x\in B_{\rho}(x_{0}): u(x,t)\leqslant \mu^{-}+ \eta\omega \right\} \right|
\leqslant \frac{7}{8}\,|B_{\rho}(x_{0})| \quad
\text{for all } \ t\in(t_{0}-\theta, t_{0}).
\end{equation}
\end{lemma}
\begin{proof}
We use inequality \eqref{eqSkr2.4} with $\chi(t)\equiv1$, $r=\rho$, $\theta$ defined in
\eqref{defthetac6c10}, and with
\begin{equation}\label{kinlem4.5}
k=\mu^{-}+\omega/2.
\end{equation}
Moreover, the second integral on the right-hand side of \eqref{eqSkr2.4} is obviously equal
to zero, since $\chi_{t}\equiv 0$, and the first, by virtue of \eqref{firstalterpsi2}
and \eqref{kinlem4.5}, is estimated in the following way:
\begin{equation}\label{upperestintBrhot0-thu-k}
\int\limits_{B_{\rho}(x_{0})\times\{t_{0}-\theta\}}
(u-k)_{-}^{2}\,\zeta^{c_{1}}\chi(t_{0}-\theta) dx\leqslant
\frac{1}{2}\left(\frac{\omega}{2}\right)^{2}|B_{\rho}(x_{0})|.
\end{equation}
By \eqref{eqSkrnew2.5}, \eqref{eqeqSkr2.8mu3} and \eqref{etaomeggeqb0},
we estimate the last term on the right-hand side of \eqref{eqSkr2.4} as follows
\begin{equation}\label{estintgAkrtheta}
\begin{aligned}
\iint\limits_{A^{-}_{k,\rho, \theta}}
g&\left( x,t, \frac{K_{1}(u-k)_{-}}{\sigma\rho\zeta} \right)
\frac{(u-k)_{-}}{\rho}\, \zeta^{c_{1}-1}dxdt
\\
&\leqslant \gamma\sigma^{\mu_{3}-1}\, e^{c_{6}\lambda(\rho)}
\frac{\omega}{\rho}\, \theta\, g(x_{0},t_{0}, \omega/2\rho)\, |B_{\rho}(x_{0})|
\\
&\leqslant \gamma\, \sigma^{\mu_{3}-1}\eta\omega^{2}\,
\frac{g(x_{0},t_{0}, \omega/2\rho)}{g(x_{0},t_{0}, \eta\omega/\rho)}\,
|B_{\rho}(x_{0})|\leqslant
\gamma\, \sigma^{\mu_{3}-1}\eta^{\mu_{3}} \Big(\frac{\omega}{2}\Big)^{2} |B_{\rho}(x_{0})|,
\end{aligned}
\end{equation}
here we assume also that $c_{1}\geqslant2-\mu_{3}$.

Using \eqref{kinlem4.5} and estimating the term on the left-hand side of \eqref{eqSkr2.4},
we obtain for any $j>1$
\begin{multline}\label{lowerestintBrhotu-k}
\int\limits_{B_{\rho}(x_{0})\times\{t\}}
(u-k)_{-}^{2}\,\zeta^{c_{1}} dx\geqslant
\int\limits_{B_{(1-\sigma)\rho}(x_{0})\times\{t\}\cap
\{u\leqslant \mu^{-}+\omega/2^{j}\}}
(u-k)_{-}^{2} dx
\\
\geqslant \left(\frac{\omega}{2}\right)^{2} \left(1-\frac{1}{2^{j-1}}\right)^{2}
\Big( |\{x\in B_{\rho}(x_{0}): u(x,t)\leqslant \mu^{-}+\omega/2^{j}\}|
-n\sigma |B_{\rho}(x_{0})| \Big).
\end{multline}
Collecting estimates \eqref{upperestintBrhot0-thu-k}, \eqref{estintgAkrtheta}
and \eqref{lowerestintBrhotu-k} in \eqref{eqSkr2.4},
we arrive at
$$
|\{x\in B_{\rho}(x_{0}): u(x,t)\leqslant \mu^{-}+\omega/2^{j}\}|\leqslant
\bigg(\frac{1}{2(1-2^{1-j})^{2}}
+ \frac{\gamma\sigma^{-\gamma}\eta^{\mu_{3}}}{(1-2^{1-j})^{2}} +n\sigma\bigg)
|B_{\rho}(x_{0})|
$$
for all $t\in (t_{0}-\theta, t_{0})$.
Choosing $j$ from the condition $(1-2^{1-j})^{2}=3/4$, then $\sigma$ from the condition
$n\sigma=1/16$, and finally, choosing $\eta$ from the condition
$\gamma\sigma^{-\gamma}\eta^{\mu_{3}}= 1/16$, we arrive at
\eqref{eqSkr4.20}. This completes the proof of the lemma.
\end{proof}

Now we will use inequality \eqref{eqSkr2.9} with $\sigma=1/2$, $r=\rho$,
$k=\mu^{-}+\varepsilon^{j}\omega$, $j=1,2,\ldots, j_{\ast}$, where
$\varepsilon=\varepsilon(\rho)\in(0,1)$ and $j_{\ast}=j_{\ast}(\rho)$ will be
determined later.
Set $v=u-\mu^{-}$,
$$A_{j}(t):=\{x\in B_{\rho}(x_{0}): v(x,t)<\varepsilon^{j}\omega\},$$
$$
Y_{j}(t):=\frac{1}{|B_{\rho}(x_{0})|} \int\limits_{A_{j}(t)}
\frac{t-t_{0}+\theta}{\theta}\,\zeta^{c_{1}}dx, \quad
y_{j}:=\sup\limits_{t_{0}-\theta/2<t<t_{0}} Y_{j}(t).
$$

Further we will assume that
\begin{equation}\label{ineq6.3}
\varepsilon^{j_{\ast}+1}\omega
\geqslant
4(1+b_{0})\rho^{1-\delta/\bar{\delta}}.
\end{equation}
Inequality \eqref{ineq6.3} together with condition \eqref{eqeqSkr2.8mu3}
and with the definition of $\theta$ in \eqref{defthetac6c10}
imply that
\begin{multline*}
\frac{1}{\theta}\Phi_{\varepsilon^{j}\omega}(x_{0},t_{0},v)\leqslant
\gamma\, \frac{\rho}{\theta}\int\limits_{0}^{\varepsilon^{j}\omega}
\frac{ds}{g\left(x_{0},t_{0}, \frac{(1+\varepsilon)\varepsilon^{j}\omega-s}{\rho}\right)}
\\
\leqslant \gamma\, \frac{\rho}{\theta}\,
\frac{\big((1+\varepsilon)\varepsilon^{j}\omega\big)^{1-\mu_{3}}}
{g\left(x_{0},t_{0}, \frac{(1+\varepsilon)\varepsilon^{j}\omega}{\rho}\right)}
\int\limits_{0}^{\varepsilon^{j}\omega}
\frac{ds}{\big((1+\varepsilon)\varepsilon^{j}\omega-s\big)^{1-\mu_{3}}}
\\
\leqslant
\gamma\,\frac{\varepsilon^{j}}{\eta}\,e^{c_{0}\lambda(\rho)}
\frac{g(x_{0},t_{0},\eta\omega/\rho)}{g(x_{0},t_{0},\eta\varepsilon^{j}\omega/\rho)}
\leqslant \gamma\, \frac{\varepsilon^{j\mu_{3}}}{\eta}\,e^{c_{0}\lambda(\rho)}
\leqslant \gamma\,e^{c_{0}\lambda(\rho)}.
\end{multline*}
Here we also used the evident inequality, which is a consequence of \eqref{eqeqSkr2.8mu3}
\begin{multline*}
G(x_{0},t_{0},{\rm w})\geqslant \int\limits_{{\rm w}/2}^{{\rm w}}g(x_{0},t_{0},s)\,ds
\\
\geqslant
c_{10}^{-1}g(x_{0},t_{0},{\rm w})\,{\rm w}^{\mu_{3}-1}
\int\limits_{{\rm w}/2}^{{\rm w}} s^{1-\mu_{3}}\,ds
\geqslant \frac{1}{2c_{10}}\,g(x_{0},t_{0},{\rm w})\,{\rm w}
\ \ \text{ if } \ {\rm w}\geqslant2(1+b_{0})R^{-\delta}.
\end{multline*}

Therefore, using Lemma \ref{lemaoneta} and the previous inequality,
from \eqref{eqSkr2.9} we obtain
\begin{multline}\label{eqparabSkr6.4}
D^{-}\int\limits_{B_{\rho}(x_{0})\times\{t\}}
\Phi_{\varepsilon^{j}\omega}(x_{0},t_{0},v)\,\frac{t-t_{0}+\theta}{\theta}\,\zeta^{c_{1}}dx
\\
+\gamma^{-1}\int\limits_{B_{\rho}(x_{0})\times\{t\}}
\left|\,\ln \frac{(1+\varepsilon)\varepsilon^{j}\omega}
{(1+\varepsilon)\varepsilon^{j}\omega-(v-\varepsilon^{j}\omega)_{-}}\right|
\,\frac{t-t_{0}+\theta}{\theta}\,\zeta^{c_{1}}dx
\\
\leqslant \gamma e^{c_{0}\lambda(\rho)}|B_{\rho}(x_{0})|
\quad \text{for all} \ \ t\in(t_{0}-\theta,t_{0}).
\end{multline}

Fix $\bar{t}\in(t_{0}-\theta/2,t_{0})$ such that $Y_{j+1}(\bar{t})=y_{j+1}$. If
$$
D^{-}\int\limits_{B_{\rho}(x_{0})\times\{\bar{t}\}}
\Phi_{\varepsilon^{j}\omega}(x_{0},t_{0},v)\,
\frac{\bar{t}-t_{0}+\theta}{\theta}\,\zeta^{c_{1}}dx\geqslant0,
$$
then inequality \eqref{eqparabSkr6.4} implies that
\begin{equation}\label{eqSkrnew6.5}
y_{j+1}\ln\frac{1}{2\varepsilon}\leqslant\gamma e^{c_{0}\lambda(\rho)}.
\end{equation}
For fixed $\nu\in(0,1)$, we choose $\varepsilon$ from the condition
$$
\varepsilon\leqslant \frac{1}{2}\exp
\left(-\frac{\gamma}{\nu} \exp\big( c_{0}(n+2)\lambda(\rho) \big)  \right),
$$
then inequality \eqref{eqSkrnew6.5} yields
$$
y_{j+1}\leqslant \nu e^{-c_{0}(n+1)\lambda(\rho)}.
$$

Assume now that
$$
D^{-}\int\limits_{B_{\rho}(x_{0})\times\{\bar{t}\}}
\Phi_{\varepsilon^{j}\omega}(x_{0},t_{0},v)\,
\frac{\bar{t}-t_{0}+\theta}{\theta}\,\zeta^{c_{1}}dx<0.
$$

\textsl{Claim}. The following inequalities hold:
\begin{equation}\label{eqSkrnew6.8}
\int\limits_{0}^{\sigma}
\frac{(\varepsilon+s)ds}{G\big(x_{0},t_{0},(\varepsilon+s)\varepsilon^{j}\omega/\rho\big)}
\geqslant
\sigma\int\limits_{0}^{1}
\frac{(\varepsilon+s)ds}{G\big(x_{0},t_{0},(\varepsilon+s)\varepsilon^{j}\omega/\rho\big)}
\quad \text{for} \ \sigma\in(0,1),
\end{equation}
\begin{equation}\label{eqSkrnew6.9}
\int\limits_{0}^{\varepsilon}
\frac{(\varepsilon+s)ds}{G\big(x_{0},t_{0},(\varepsilon+s)\varepsilon^{j}\omega/\rho\big)}
\leqslant
2c_{10}^{2}\,\varepsilon^{\mu_{3}}\int\limits_{0}^{1}
\frac{(\varepsilon+s)ds}{G\big(x_{0},t_{0},(\varepsilon+s)\varepsilon^{j}\omega/\rho\big)}.
\end{equation}

Indeed, inequality \eqref{eqSkrnew6.8} is a consequence of the fact that the function
$G(x_{0},t_{0}, {\rm v})/{\rm v}$ is nondecreasing
$$
\int\limits_{0}^{\sigma}
\frac{(\varepsilon+s)ds}{G\big(x_{0},t_{0},(\varepsilon+s)\varepsilon^{j}\omega/\rho\big)}=
\sigma
\int\limits_{0}^{1}
\frac{(\varepsilon+s\sigma)ds}{G\big(x_{0},t_{0},(\varepsilon+s\sigma)\varepsilon^{j}\omega/\rho\big)}
\geqslant
\sigma\int\limits_{0}^{1}
\frac{(\varepsilon+s)ds}{G\big(x_{0},t_{0},(\varepsilon+s)\varepsilon^{j}\omega/\rho\big)}.
$$
To prove \eqref{eqSkrnew6.9} we use condition \eqref{eqeqSkr2.8mu3}:
\begin{multline*}
\int\limits_{0}^{\varepsilon}
\frac{(\varepsilon+s)ds}{G\big(x_{0},t_{0},(\varepsilon+s)\varepsilon^{j}\omega/\rho\big)}=
\varepsilon^{2}
\int\limits_{0}^{1}
\frac{(1+s)ds}{G\big(x_{0},t_{0},(1+s)\varepsilon^{j+1}\omega/\rho\big)}
\\
\leqslant
2c_{10}^{2}\,\varepsilon^{2}
\int\limits_{0}^{1}
\frac{1+s}{G\big(x_{0},t_{0},(\varepsilon+s)\varepsilon^{j}\omega/\rho\big)}
\left( \frac{\varepsilon+s}{\varepsilon(1+s)} \right)^{2-\mu_{3}} ds
\leqslant
2c_{10}^{2}\,\varepsilon^{\mu_{3}} \int\limits_{0}^{1}
\frac{(\varepsilon+s)ds}{G\big(x_{0},t_{0},(\varepsilon+s)\varepsilon^{j}\omega/\rho\big)},
\end{multline*}
here we also used the inequality
\begin{multline*}
G(x_{0},t_{0}, {\rm w})\leqslant g(x_{0},t_{0}, {\rm w}) {\rm w}
\leqslant
c_{10}\,\frac{g(x_{0},t_{0}, {\rm v}) {\rm w}^{2}}{{\rm v}}
\left( \frac{{\rm v}}{{\rm w}} \right)^{\mu_{3}}
\\
\leqslant
2c_{10}^{2}\,G(x_{0},t_{0}, {\rm v})
\left( \frac{{\rm w}}{{\rm v}} \right)^{2-\mu_{3}}
 \ \ \text{if } \ {\rm w}\geqslant {\rm v}\geqslant 2(1+b_{0})R^{-\delta},
\end{multline*}
which proves the claim.

Define
$$
t_{\ast}:=\sup\Bigg\{ t\in(t_{0}-\theta/2,t_{0}): D^{-}\int\limits_{B_{\rho}(x_{0})\times\{t\}}
\Phi_{\varepsilon^{j}\omega}(x_{0},t_{0}, v)\,
\frac{t-t_{0}+\theta}{\theta}\,\zeta^{c_{1}} dx\geqslant0 \Bigg\}.
$$
From the definition of $t_{\ast}$ we have
\begin{equation}\label{eqSkr6.10new}
I(\bar{t}):=\int\limits_{B_{\rho}(x_{0})\times\{\bar{t}\}}
\Phi_{\varepsilon^{j}\omega}(x_{0},t_{0}, v)\,
\frac{\bar{t}-t_{0}+\theta}{\theta}\,\zeta^{c_{1}} dx
\leqslant I(t_{\ast}).
\end{equation}
By \eqref{eqSkrnew6.9} we obtain
\begin{multline}\label{eqSkr6.11new}
I(\bar{t})\geqslant \int\limits_{A_{j+1}(\bar{t})}
\frac{\bar{t}-t_{0}+\theta}{\theta}\,\zeta^{c_{1}}dx
\int\limits_{0}^{\varepsilon^{j}(1-\varepsilon)\omega}
\frac{(1+\varepsilon)\varepsilon^{j}\omega-s}
{G\left(x_{0},t_{0}, \frac{(1+\varepsilon)\varepsilon^{j}\omega-s}{\rho}\right)}\,ds
\\
=(\varepsilon^{j}\omega)^{2}\int\limits_{A_{j+1}(\bar{t})}
\frac{\bar{t}-t_{0}+\theta}{\theta}\,\zeta^{c_{1}}dx
\int\limits_{0}^{1-\varepsilon}
\frac{(1+\varepsilon-s)ds}
{G\left(x_{0},t_{0},\frac{1+\varepsilon-s}{\rho}\varepsilon^{j}\omega\right)}
\\
=(\varepsilon^{j}\omega)^{2}\int\limits_{A_{j+1}(\bar{t})}
\frac{\bar{t}-t_{0}+\theta}{\theta}\,\zeta^{c_{1}}dx
\int\limits_{\varepsilon}^{1}
\frac{(\varepsilon+s)ds}
{G\left(x_{0},t_{0},\frac{\varepsilon+s}{\rho}\varepsilon^{j}\omega\right)}
\\
\geqslant y_{j+1}(1-2c_{10}^{2}\varepsilon^{\mu_{3}})(\varepsilon^{j}\omega)^{2}
|B_{\rho}(x_{0})| \int\limits_{0}^{1}
\frac{(\varepsilon+s)ds}
{G\left(x_{0},t_{0},\frac{\varepsilon+s}{\rho}\varepsilon^{j}\omega\right)}.
\end{multline}

Let us estimate from above the term on the right-hand side of \eqref{eqSkr6.10new}.
By Fubini's theorem we conclude that
\begin{multline*}
I(t_{\ast})\leqslant \int\limits_{B_{\rho}(x_{0})\times\{t_{\ast}\}}
\frac{t_{\ast}-t_{0}+\theta}{\theta}\,\zeta^{c_{1}}dx
\int\limits_{0}^{\varepsilon^{j}\omega}
\frac{\big( (1+\varepsilon)\varepsilon^{j}\omega-s \big)
\mathbb{I}_{\{\varepsilon^{j}\omega-v>s\}}}
{G\big(x_{0},t_{0}, \frac{(1+\varepsilon)\varepsilon^{j}\omega-s}{\rho}\big)}\,ds
\\
=(\varepsilon^{j}\omega)^{2}\int\limits_{0}^{1}
\frac{(1+\varepsilon-s)ds}
{G\big(x_{0},t_{0}, \frac{1+\varepsilon-s}{\rho}\,\varepsilon^{j}\omega\big)}
\int\limits_{B_{\rho}(x_{0})\times\{t_{\ast}\}}
\frac{t_{\ast}-t_{0}+\theta}{\theta}\,\zeta^{c_{1}}
\mathbb{I}_{\{\varepsilon^{j}\omega-v>\varepsilon^{j}s\}} dx.
\end{multline*}

Similarly to \eqref{eqSkrnew6.5} we obtain
$$
\int\limits_{B_{\rho}(x_{0})\times\{t_{\ast}\}}
\frac{t_{\ast}-t_{0}+\theta}{\theta}\,\zeta^{c_{1}}
\mathbb{I}_{\{\varepsilon^{j}\omega-v>\varepsilon^{j}s\}} dx
\leqslant
\frac{\gamma\, e^{c_{0}\lambda(\rho)}}{\ln \frac{1+\varepsilon}{1+\varepsilon-s}}\,
|B_{\rho}(x_{0})| \quad \text{for any} \ s\in[0,1].
$$
Particularly, if
$$
(1+\varepsilon)
\bigg[1-\exp\Big(-\frac{2\gamma}{\nu}\,e^{c_{0}(n+2)\lambda(\rho)}\Big)\bigg]<s<1,
$$
then
$$
\int\limits_{B_{\rho}(x_{0})\times\{t_{\ast}\}}
\frac{t_{\ast}-t_{0}+\theta}{\theta}\,\zeta^{c_{1}}
\mathbb{I}_{\{\varepsilon^{j}\omega-v>\varepsilon^{j}s\}} dx
\leqslant
\frac{\nu}{2}\, e^{-c_{0}(n+1)\lambda(\rho)}|B_{\rho}(x_{0})|.
$$
Choosing $s_{\ast}$ from the condition
$$
s_{\ast}=(1+\varepsilon)
\bigg[1-\exp\Big(-\frac{4\gamma}{\nu}\,e^{c_{0}(n+2)\lambda(\rho)}\Big)\bigg],
$$
and assuming that
$
y_{j}\geqslant \nu e^{-c_{0}(n+1)\lambda(\rho)},
$
from the previous and \eqref{eqSkrnew6.8}, we obtain
\begin{multline}\label{eqSkr6.12new}
I(t_{\ast})\leqslant y_{j}(\varepsilon^{j}\omega)^{2}|B_{\rho}(x_{0})|
\bigg( \int\limits_{0}^{s_{\ast}}
\frac{(1+\varepsilon-s)ds}
{G\big(x_{0},t_{0}, \frac{1+\varepsilon-s}{\rho} \varepsilon^{j}\omega\big)}
+\frac{1}{2} \int\limits_{s_{\ast}}^{1}
\frac{(1+\varepsilon-s)ds}
{G\big(x_{0},t_{0}, \frac{1+\varepsilon-s}{\rho} \varepsilon^{j}\omega\big)} \bigg)
\\
\leqslant y_{j}(\varepsilon^{j}\omega)^{2}|B_{\rho}(x_{0})|
\bigg( \int\limits_{0}^{1} \frac{(\varepsilon+s)ds}
{G\big(x_{0},t_{0}, \frac{\varepsilon+s}{\rho} \varepsilon^{j}\omega\big)}
-\frac{1}{2}
\int\limits_{0}^{1-s_{\ast}} \frac{(\varepsilon+s)ds}
{G\big(x_{0},t_{0}, \frac{\varepsilon+s}{\rho} \varepsilon^{j}\omega\big)} \bigg)
\\
\leqslant \bigg(1-\frac{1-s_{\ast}}{2}\bigg)
y_{j}(\varepsilon^{j}\omega)^{2}|B_{\rho}(x_{0})|
\int\limits_{0}^{1} \frac{(\varepsilon+s)ds}
{G\big(x_{0},t_{0}, \frac{\varepsilon+s}{\rho} \varepsilon^{j}\omega\big)}.
\end{multline}
Combining \eqref{eqSkr6.10new}, \eqref{eqSkr6.11new}, \eqref{eqSkr6.12new},
we obtain that either
$$
y_{j}\leqslant \nu e^{-c_{0}(n+1)\lambda(\rho)}
$$
or
\begin{equation}\label{eqSkr6.13new}
y_{j+1}\leqslant \frac{1-(1-s_{\ast})/2}{1-2c_{10}^{2}\varepsilon^{\mu_{3}}}y_{j}.
\end{equation}
By our choice of $s_{\ast}$ we have
\begin{multline*}
\frac{1-(1-s_{\ast})/2}{1-2c_{10}^{2}\varepsilon^{\mu_{3}}}=1-
\frac{(1+\varepsilon)\exp(-\frac{4\gamma}{\nu}e^{c_{0}(n+2)\lambda(\rho)})-4c_{10}^{2}\varepsilon^{\mu_{3}}-\varepsilon}
{2(1-2c_{10}^{2}\varepsilon^{\mu_{3}})}
\\
\leqslant
1-\frac{1}{2}\left( \exp\left(-\frac{4\gamma}{\nu}e^{c_{0}(n+2)\lambda(\rho)}\right)
-(1+4c_{10}^{2})\varepsilon^{\mu_{3}} \right).
\end{multline*}
Choosing $\varepsilon$ from the condition
$$
(1+4c_{10}^{2})\varepsilon^{\mu_{3}}=\frac{1}{2}
\exp\left(-\frac{4\gamma}{\nu}e^{c_{0}(n+2)\lambda(\rho)}\right),
$$
we obtain from \eqref{eqSkr6.13new} that
$$
y_{j+1}\leqslant \big(1-\sigma_{0}(\rho)\big)y_{j}, \quad
\sigma_{0}(\rho)=\frac{1}{4}\exp\left(-\frac{4\gamma}{\nu}e^{c_{0}(n+2)\lambda(\rho)}\right).
$$
Iterating this inequality we obtain that
$$
y_{j_{\ast}}=\big(1-\sigma_{0}(\rho)\big)^{j_{\ast}}y_{0}
\leqslant \big(1-\sigma_{0}(\rho)\big)^{j_{\ast}},
$$
choosing $j_{\ast}$ from the condition
$$
\big(1-\sigma_{0}(\rho)\big)^{j_{\ast}}\leqslant
\nu e^{-c_{0}(n+1)\lambda(\rho)},
$$
we conclude that
\begin{equation}\label{eqSkr6.15new}
y_{j_{\ast}}\leqslant\nu e^{-c_{0}(n+1)\lambda(\rho)}.
\end{equation}

Let us estimate the value of $\varepsilon^{j_{\ast}}$ from below.
By our choices of $\varepsilon$ and $j_{\ast}$ we have
\begin{multline*}
\varepsilon^{j_{\ast}}\geqslant\exp
\left( -\frac{1}{\sigma_{0}(\rho)}\ln\Big(\frac{1}{\nu}e^{c_{0}(n+2)\lambda(\rho)}\Big)
\ln\Big( 6^{1/\mu_{3}}\exp \Big( \frac{4\gamma}{\nu\mu_{3}}
e^{c_{0}(n+2)\lambda(\rho)} \Big)  \Big) \right)
\\
\geqslant \gamma_{0}^{-1}
\exp \left( -6\exp\Big( \frac{4\gamma}{\nu} e^{c_{0}(n+2)\lambda(\rho)} \Big)  \right)
=\gamma_{0}^{-1} \exp\big(-6\Lambda_{1}(C,\beta,\rho)\big),
\end{multline*}
where $C=4\gamma/\nu$, $\beta=c_{0}(n+2)$ and
$\Lambda_{1}(C,\beta,\rho):=\exp(Ce^{\beta\lambda(\rho)})$.
Particularly, if
\begin{equation}\label{eqSkr6.16new}
\omega\geqslant \frac{2\gamma_{0}^{2}}{2\gamma_{0}-1}
(1+b_{0})\rho^{1-\delta/\bar{\delta}}
\exp\big(8\Lambda_{1}(C,\beta,\rho)  \big),
\end{equation}
then $\omega\varepsilon^{j_{\ast}+1}\geqslant (1+b_{0})\rho^{1-\delta/\bar{\delta}}$
and inequalities \eqref{etaomeggeqb0} and \eqref{ineq6.3} be fulfilled.
Moreover, by Theorem \ref{thSkr3.1} from \eqref{eqSkr6.15new} and \eqref{eqSkr6.16new}
we obtain that
$$
u(x,t)\geqslant \mu^{-}+\frac{1}{2}\tau_{1}(\rho)
\quad \text{for a.a.} \ (x,t)\in Q_{\rho/2,\theta'/2}(x_{0},t_{0})
$$
where
$$
\tau_{1}(\rho):=\gamma_{0}^{-1}
\exp\big(-8\Lambda_{1}(C,\beta,\rho) \big),
\quad \theta':=\frac{\rho^{2}}{\psi(x_{0},t_{0},\omega\tau_{1}(\rho)/\rho)},
$$
which implies
\begin{equation*}
\text{osc}\{u; Q_{\rho/2,\theta'/2}(x_{0},t_{0})\}
\leqslant \bigg(1-\frac{1}{2}\tau_{1}(\rho)\bigg)\omega.
\end{equation*}

Set
$$
\rho_{1}:=\frac{1}{2}\bigg( 1-\frac{1}{2\gamma_{0}} \bigg)
 \rho\,\tau_{1}(\rho),
\quad
\omega_{1}:=\bigg(1-\frac{1}{2}\tau_{1}(\rho)\bigg)\omega,
$$
$$
\theta_{1}:=\frac{\rho_{1}^{2}}{\psi(x_{0},t_{0},\omega_{1}/\rho_{1})},
\quad
Q_{1}:=Q_{\rho_{1},\theta_{1}}(x_{0},t_{0}).
$$
By (${\rm g}_{2}\psi_{2}$) we have 
$
g(x_{0},t_{0}, \omega_{1}/\rho_{1}) \geqslant
g(x_{0},t_{0}, \omega\tau_{1}(\rho)/\rho)
$
and $\theta'\geqslant \theta_{1}$,
so the inclusion
$Q_{\rho/2, \theta'/2}(x_{0},t_{0})\subset Q_{1}$ follows.
Moreover, we obtain $\text{osc}\{u;Q_{1}\}\leqslant \omega_{1}$.

For $j=1,2,\ldots$, we define the sequences
$$
\begin{aligned}
\rho_{j}:&=\frac{1}{2}\bigg( 1-\frac{1}{2\gamma_{0}} \bigg)
\rho_{j-1} \tau_{1}(\rho_{j-1}),
\\
\omega_{j}:&=\max\left\{ \bigg(1-\frac{1}{2}\tau_{1}(\rho_{j-1}) \bigg)\omega_{j-1}, \
\frac{2\gamma_{0}^{2}(1+b_{0})}{2\gamma_{0}-1}\,
\frac{\rho_{j-1}^{1-\delta/\bar{\delta}}}{\tau_{1}(\rho_{j-1})} \right\},
\\
\theta_{j}:&=\frac{\rho_{j}^{2}}{\psi(x_{0},t_{0},\omega_{j}/\rho_{j})},
\quad \quad \quad \quad \hskip 3mm Q_{j}:=Q_{\rho_{j}, \theta_{j}}(x_{0},t_{0}),
\\
\widetilde{\theta}_{j}:&=
\frac{\rho_{j}^{2}}
{\psi\left(x_{0},t_{0},(1+b_{0})\rho_{j-1}^{-\delta/\bar{\delta}}\right)},
\quad
\widetilde{Q}_{j}:=Q_{\rho_{j}, \widetilde{\theta}_{j}}(x_{0},t_{0}).
\end{aligned}
$$
Repeating the previous procedure, we obtain for any $j=1,2,\ldots$ that
\begin{multline*}
\text{osc}\{u;\widetilde{Q}_{j}\}\leqslant\text{osc}\{u;Q_{j}\} \leqslant \omega_{j}\leqslant
\left(1-\frac{1}{2}\,\tau_{1}(\rho_{j-1}) \right)\omega_{j-1}+
\frac{\gamma(1+b_{0})\rho_{j-1}^{1-\delta/\bar{\delta}}}{\tau_{1}(\rho_{j-1})}
\\
\leqslant \omega\prod\limits_{i=0}^{j-1}\left(1-\frac{1}{2}\tau_{1}(\rho_{i}) \right)
+\frac{\gamma(1+b_{0})\rho_{j-1}^{1-\delta/\bar{\delta}}}{\tau_{1}(\rho_{j-1})}+
\gamma(1+b_{0})\sum\limits_{i=0}^{j-2}\frac{\rho_{i}^{1-\delta/\bar{\delta}}}{\tau_{1}(\rho_{i})}
\prod\limits_{k=i+1}^{j-1}\left(1-\frac{1}{2}\tau_{1}(\rho_{k}) \right)
\phantom{\leqslant\leqslant}
\\
\leqslant \omega\exp\left( -\frac{1}{2}\sum\limits_{i=0}^{j-1}\tau_{1}(\rho_{i}) \right)
+\gamma(1+b_{0})\rho^{1-\delta_{0}}
\exp\big( 8\Lambda_{1}(c,\beta,\rho) \big),
\end{multline*}
which proves Theorem \ref{th2.2}.

\vskip3.5mm
{\bf Acknowledgements.} The research of the first author was supported by grants of Ministry of Education and Science of Ukraine
(project numbers are 0118U003138, 0119U100421).

\bigskip

CONTACT INFORMATION

\medskip
Igor I.~Skrypnik\\Institute of Applied Mathematics and Mechanics,
National Academy of Sciences of Ukraine, Gen. Batiouk Str. 19, 84116 Sloviansk, Ukraine\\
Vasyl' Stus Donetsk National University, Mathematical Analysis and Differential Equations,
600-richcha Str. 21, 21021 Vinnytsia, Ukraine\\iskrypnik@iamm.donbass.com

\medskip
Mykhailo V.~Voitovych\\Institute of Applied Mathematics and Mechanics,
National Academy of Sciences of Ukraine, Gen. Batiouk Str. 19, 84116 Sloviansk, Ukraine\\voitovichmv76@gmail.com


\begin{thebibliography}{99}

\bibitem{AcerbiFuscoJDE94}
E.~Acerbi, N.~Fusco,
Partial regularity under anisotropic $(p,q)$ growth conditions,
J. Differential Equations \textbf{107} (1994), no.~1, 46--67.


\bibitem{AcerbiMingioneArchRat01}
E.~Acerbi, G.~Mingione,
Regularity results for a class of functionals with non-standard growth,
 Arch. Ration. Mech. Anal. \textbf{156} (2001), no.~2, 121--140.


\bibitem{AcerbiMingioneAnSc01}
E.~Acerbi, G.~Mingione,
Regularity results for a class of quasiconvex functionals with nonstandard growth,
Ann. Scuola Norm. Sup. Pisa Cl. Sci. (4), \textbf{30} (2001), no.~2, 311--339.


\bibitem{AcerbiMingioneArchRat02}
E.~Acerbi, G.~Mingione,
Regularity results for stationary electro-rheological fluids,
Arch. Ration. Mech. Anal. \textbf{164} (2002), no.~3, 213--259.


\bibitem{AcerbiMingioneJRAngMath05}
E.~Acerbi, G.~Mingione,
Gradient estimates for the $p(x)$-Laplacean system,
J. Reine Angew. Math. \textbf{584} (2005), 117--148.


\bibitem{Alhutov97}
Yu.\,A.~Alkhutov,
The Harnack inequality and the Holder property of solutions of nonlinear elliptic equations
with a nonstandard growth condition (Russian),
Differ. Uravn. \textbf{33} (1997), no.~12, 1651--1660;
translation in Differential Equations \textbf{33} (1997), no.~12, 1653--1663 (1998).



\bibitem{AlhutovKrash04}
Yu.\,A.~Alkhutov, O.\,V.~Krasheninnikova,
Continuity at boundary points of solutions of quasilinear
elliptic equations with a nonstandard growth condition (Russian),
Izv. Ross. Akad. Nauk Ser. Mat. \textbf{68} (2004), no.~6, 3--60;
translation in Izv. Math. \textbf{68} (2004), no.~6, 1063--1117.



\bibitem{AlhutovKrash08}
Yu.\,A.~Alkhutov, O.\,V.~Krasheninnikova,
On the continuity of solutions of elliptic equations with a variable order of nonlinearity
(Russian),
Tr. Mat. Inst. Steklova \textbf{261}, (2008), Differ. Uravn. i Din. Sist., 7--15;
translation in Proc. Steklov Inst. Math. \textbf{261} (2008), no.~1--10.



\bibitem{AlkhSurnAlgAn19}
Yu.\,A.~Alkhutov, M.\,D.~Surnachev,
Behavior at a boundary point of solutions of the Dirichlet problem for the $p(x)$-Laplacian
(Russian), Algebra i Analiz \textbf{31} (2019), no. 2, 88--117;
translation in St. Petersburg Math. J. \textbf{31} (2020), no. 2, 251--271.



\bibitem{AlkhSurnApAn19}
Yu.\,A.~Alkhutov, M.\,D.~Surnachev,
A Harnack inequality for a transmission problem with $p(x)$-Laplacian,
Appl. Anal., \textbf{98} (2019), no.~1-2, 332--344.



\bibitem{AntDiazShm2002_monogr}
S.\,N.~Antontsev, J.\,I.~D\'{\i}az, S.~Shmarev,
Energy Methods for Free Boundary Problems. Applications to Nonlinear PDEs and Fluid Mechanics,
in: Progress in Nonlinear Differential Equations and their Applications, vol. 48,
Birkhauser Boston, Inc., Boston, MA, 2002.



\bibitem{BarColMing}
P.~Baroni, M.~Colombo, G.~Mingione,
Harnack inequalities for double phase functionals,
Nonlinear Anal. \textbf{121} (2015), 206--222.



\bibitem{BarColMingStPt16}
P.~Baroni, M.~Colombo, G.~Mingione,
Non-autonomous functionals, borderline cases and related
function classes,
St. Petersburg Math. J. \textbf{27} (2016), 347--379.



\bibitem{BarColMingCalc.Var.18}
P. Baroni, M. Colombo, G. Mingione,
Regularity for general functionals with double phase, Calc. Var.
Partial Differential Equations \textbf{57} (2018), Paper No. 62, 48 pp.




\bibitem{BenedMascoloAbsApplAn04}
I.~Benedetti, E.~Mascolo,
Regularity of minimizers for nonconvex vectorial integrals with p-q growth via relaxation methods,
Abstr. Appl. Anal. (2004), no.~1, 27--44.



\bibitem{DiBenedetto86}
E.~Di\,Benedetto, On the local behaviour of solutions of degenerate parabolic equations
with measurable coefficients,
Ann. Sc. Norm. Super. Pisa Cl. Sci. (4), \textbf{13} (1986), no.~3, 487--535.



\bibitem{DiBenedettoDegParEq}
E.~Di\,Benedetto, Degenerate Parabolic Equations, Springer-Verlag, New York, 1993.




\bibitem{BurSkr_JEvolEq}
K.\,O.~Buryachenko, I.\,I.~Skrypnik,
Harnack's inequality for double-phase parabolic equations,
Potential Analysis (to appear).



\bibitem{ByunOh17}
S.-S. Byun, J. Oh,
Global gradient estimates for non-uniformly elliptic equations,
Calc. Var. Partial Differential Equations \textbf{56} (2017), no. 2, Paper No. 46, 36 pp.



\bibitem{ByunRyuShin18}
S.-S. Byun, S. Ryu, P. Shin,
Calderon-Zygmund estimates for $\omega$-minimizers of double phase variational
problems, Appl. Math. Letters \textbf{86} (2018), 256--263.



\bibitem{ChiadoPiatCoscia}
V.~Chiad\`{o} Piat, A.~Coscia,
H\"{o}lder continuity of minimizers of functionals with variable growth exponent,
Manuscripta Math. \textbf{93} (1997), no.~3, 283--299.



\bibitem{ColMing15}
M.~Colombo, G.~Mingione, Regularity for double phase variational problems,
Arch. Rational Mech. Anal.  \textbf{215} (2015), No. 2, 443--496.



\bibitem{ColMing218}
M.~Colombo, G.~Mingione,
Bounded minimisers of double phase variational integrals,
Arch. Rational Mech. Anal.  \textbf{218} (2015), No. 1, 219--273.



\bibitem{ColMingJFnctAn16}
M. Colombo, G. Mingione,
Calderon-Zygmund estimates and non-uniformly elliptic operators,
J. Funct. Anal. \textbf{270} (2016), 1416--1478.



\bibitem{ElMarcMas16}
M. Eleuteri, P. Marcellini, E. Mascolo,
Lipschitz continuity for energy integrals with variable exponents,
Atti Accad. Naz. Lincei Rend. Lincei Mat. Appl. \textbf{27} (2016), no. 1, 61--87.



\bibitem{ElMarcMasPuraAppl16}
M. Eleuteri, P. Marcellini, E. Mascolo,
Lipschitz estimates for systems with ellipticity conditions at infinity,
Ann. Mat. Pura Appl. (4) \textbf{195} (2016), no. 5, 1575--1603.



\bibitem{ElMarcMasAdvCalc17}
M. Eleuteri, P. Marcellini, E. Mascolo,
Regularity for scalar integrals without structure conditions,
Adv. Calc. Var., to appear.
https://doi.org/10.1515/acv-2017-0037



\bibitem{GiandiNap13}
F. Giannetti, A. Passarelli di Napoli,
Regularity results for a new class of functionals with non-standard growth conditions,
J. Differential Equations \textbf{254} (2013) 1280--1305.



\bibitem{HarHastOrlicz}
P. Harjulehto, P. H\"{a}st\"{o},
Orlicz Spaces and Generalized Orlicz Spaces, Lecture Notes in Mathematics,
vol. 2236, Springer, Cham, 2019, X+169 pages. DOI: 10.1007/978-3-030-15100-3.



\bibitem{HarHastLee18}
P. Harjulehto, P. H\"{a}st\"{o}, M. Lee,
H\"{o}lder continuity of quasiminimizers and $\omega$-minimizers of functionals with generalized
Orlicz growth,  arXiv:1906.01866v2 [math.AP].



\bibitem{HarHastToiv17}
P. Harjulehto, P. H\"{a}st\"{o}, O. Toivanen,
H\"{o}lder regularity of quasiminimizers under generalized growth conditions,
Calc. Var. Partial Differential Equations \textbf{56} (2017), no. 2, Art. 22, 26 pp.


\bibitem{HastOkarXiv19}
P. H\"{a}st\"{o}, J. Ok, Maximal regularity for local minimizers of non-autonomous functionals,
arXiv:1902.00261v2 [math.AP].






\bibitem{HwangLieberman287}
S. Hwang, G.\,M. Lieberman, H\"{o}lder continuity of bounded weak solutions to generalized
parabolic $p$-Laplacian equations I: degenerate case,
Electron. J. Differential Equations, \textbf{2015} (2015), no. 287, 1--32.


\bibitem{HwangLieberman288}
S. Hwang, G.\,M. Lieberman, H\"{o}lder continuity of bounded weak solutions to generalized
parabolic $p$-Laplacian equations II: singular case,
Electron. J. Differential Equations, \textbf{2015} (2015), no. 288, 1--24.



\bibitem{LadUr}
O.\,A.~Ladyzhenskaya, N.\,N.~Ural'tseva,
Linear and quasilinear elliptic equations,
Nauka, Moscow, 1973.


\bibitem{Lieberman91}
G.\,M.~Lieberman,
The natural generalization of the natural conditions of
Ladyzhenskaya and Ural'tseva for elliptic equations,
Comm. Partial Differential Equations \textbf{16} (1991), no. 2-3, 311--361.



\bibitem{Marcellini1989}
P.~Marcellini, Regularity of minimizers of integrals of the calculus of variations with
non standard growth conditions, Arch. Rational Mech. Anal.  \textbf{105} (1989), no.~3, 267--284.



\bibitem{Marcellini1991}
P.~Marcellini, Regularity and existence of solutions of elliptic equations with $p,q$-growth conditions,
J. Differential Equations \textbf{90} (1991), No. 1, 1--30.



\bibitem{Ok16}
J. Ok,
Gradient estimates for elliptic equations with $L^{p(\cdot)}\log L$ growth,
Calc. Var. Partial Differential Equations \textbf{55} (2016), no. 2, 1--30.



\bibitem{OkJMAnAppl16}
J. Ok,
Regularity results for a class of obstacle problems with nonstandard growth,
J. Math. Anal. Appl. \textbf{444} (2016), no. 2, 957--979.



\bibitem{OkJMAdvNonAn18}
J. Ok,
Harnack inequality for a class of functionals with non-standard growth via De Giorgi's method,
Adv. Nonlinear Anal. \textbf{7} (2018), no. 2, 167--182.



\bibitem{OkCalcVar17}
J. Ok,
Regularity of $\omega$-minimizers for a class of functionals with non-standard growth,
Calc. Var. Partial Differential Equations \textbf{56} (2017), no. 2, Art. 48, 31 pp.



\bibitem{Ruzicka2000}
M.~R\r{u}\v{z}i\v{c}ka,
Electrorheological fluids: modeling and mathematical theory, in:
Lecture Notes in Mathematics, vol.~1748, Springer-Verlag, Berlin, 2000.


\bibitem{SkrVoitUMB19}
I.\,I. Skrypnik, M.\,V. Voitovych,
$\mathfrak{B}_{1}$ classes of De Giorgi, Ladyzhenskaya and Ural'tseva
and their application to elliptic and parabolic equations with nonstandard growth,
Ukr. Mat. Visn. \textbf{16} (2019), no. 3, 403--447.



\bibitem{WangLiuZhao19}
B. Wang, D. Liu, P. Zhao,
H\"{o}lder continuity for nonlinear elliptic problem in Musielak-Orlicz-Sobolev space,
 J. Differential Equations \textbf{266} (2019), No. 8, 4835--4863.



\bibitem{Weickert}
J.~Weickert,
Anisotropic diffusion in image processing,
European Consortium for Mathematics in Industry, B. G. Teubner, Stuttgart, 1998.



\bibitem{ZhangRadul18}
Q. Zhang, V. R\u{a}dulescu,
Double phase anisotropic variational problems and combined effects of reaction
and absorption terms, J. Math. Pures Appl. (9) \textbf{118} (2018), 159--203.


\bibitem{ZhikIzv1983}
V.\,V.~Zhikov, Questions of convergence, duality and averaging for functionals
of the calculus of variations (Russian),
Izv. Akad. Nauk SSSR Ser. Mat. \textbf{47} (1983), no.~5, 961--998.



\bibitem{ZhikIzv1986}
V.\,V.~Zhikov,
Averaging of functionals of the calculus of variations and elasticity theory
(Russian),
Izv. Akad. Nauk SSSR Ser. Mat. \textbf{50}, (1986), no.~4, 675--710, 877.



\bibitem{ZhikJDiffEq91}
V.\,V.~Zhikov,
The Lavrent'ev effect and averaging of nonlinear variational problems,
(Russian),
Differentsial'nye Uravneniya \textbf{27} (1991), no.~1, 42--50, 180;
translation in
Differential Equations \textbf{27} (1991), no.~1, 32--39.



\bibitem{ZhikJMathPh94}
V.\,V.~Zhikov,
On Lavrentiev's phenomenon,
Russian J. Math. Phys. \textbf{3} (1995), no.~2, 249--269.


\bibitem{ZhikJMathPh9798}
V.\,V.~Zhikov,
On some variational problems,
Russian J. Math. Phys. \textbf{5} (1997), no. 1, 105--116 (1998).



\bibitem{ZhikPOMI04}
V.\,V.~Zhikov,  On the density of smooth functions in Sobolev-Orlicz spaces.
(Russian) Zap. Nauchn. Sem. S.-Peterburg. Otdel. Mat. Inst. Steklov. (POMI) 310 (2004),
Kraev. Zadachi Mat. Fiz. i Smezh. Vopr. Teor. Funkts. 35 [34], 67--81, 226;
translation in J. Math. Sci. (N.Y.) 132 (2006), no. 3, 285--294.


\bibitem{ZhikKozlOlein94}
V.\,V.~Zhikov, S.\,M.~Kozlov, O.\,A.~Oleinik,
Homogenization of differential operators and integral functionals,
Springer-Verlag, Berlin, 1994.


\end{thebibliography}
\end{document}